\documentclass[12pt]{amsart}

\usepackage{times}
\usepackage{amssymb}
\usepackage{graphicx,xspace}
\usepackage{enumitem}

\usepackage[T1]{fontenc}
\usepackage[utf8]{inputenc}

\usepackage{hyperref}

\theoremstyle{plain}

\newtheorem{lemma}{Lemma}[section]
\newtheorem{theorem}[lemma]{Theorem}
\newtheorem{corollary}[lemma]{Corollary}
\newtheorem{proposition}[lemma]{Proposition}
\newtheorem{problem}[lemma]{Problem}
\newtheorem{conjecture}[lemma]{Conjecture}

\theoremstyle{remark}

\newcommand{\R}{\mathbb{R}}
\newcommand{\C}{\mathbb{C}}

\newcommand{\Y}{\mathbb{Y}}
\newcommand{\X}{\mathcal{X}}
\newcommand{\F}{\mathcal{F}}
\newcommand{\p}{\mathbf{p}}
\newcommand{\q}{\mathbf{q}}
\newcommand{\Sym}{\mathfrak{S}}
\newcommand{\stirling}[2]{\genfrac{\{}{\}}{0pt}{}{#1}{#2}}
\newcommand{\I}{\mathcal{I}}
\newcommand{\M}{\mathcal{M}}
\newcommand{\V}{\mathcal{V}}
\newcommand{\K}{\mathcal{K}}
\newcommand{\NN}{Q}      
\newcommand{\ee}{e$^2$}
\newcommand{\eee}{e$^3$}
\newcommand{\eeee}{e$^4$}
\newcommand{\eeeee}{e$^5$}

\DeclareMathOperator{\Tr}{Tr}

\DeclareMathOperator{\contents}{contents}
\DeclareMathOperator{\id}{id}
\DeclareMathOperator{\NC}{NC}
\DeclareMathOperator{\Bad}{Bad}
\DeclareMathOperator{\trivial}{trivial}
\DeclareMathOperator{\dimension}{dim}
\DeclareMathOperator{\Imag}{Im}

\author{Maciej Dołęga}
\address{Institute of Mathematics,
University of Wroclaw,  \mbox{pl.\ Grunwaldzki~2/4,} 50-384
Wroclaw, Poland} 
\email{Maciej.Dolega@math.uni.wroc.pl}

\author{Valentin Féray}
\address{LaBRI, University Bordeaux 1, 351 cours de la Libération, 33400
Talence, France}
\email{feray@labri.fr}

\author{Piotr \'Sniady}
\address{Institute of Mathematics, Polish Academy of Sciences,
ul.~Śniadec\-kich 8, 00-956 Warszawa, Poland \newline
\indent Institute of Mathematics,
University of Wroclaw,  \mbox{pl.\ Grunwaldzki~2/4,} 50-384
Wroclaw, Poland}
\email{Piotr.Sniady@math.uni.wroc.pl}

\title[Explicit combinatorial interpretation of Kerov 
polynomials]{Explicit combinatorial interpretation \\ of Kerov character
polynomials  \\ as numbers of permutation factorizations}

\begin{document}

\begin{abstract}
We find an explicit combinatorial interpretation of the coefficients
of Kerov character polynomials which express the value of normalized
irreducible characters of the symmetric groups $\Sym(n)$ in terms of free
cumulants $R_2,R_3,\dots$ of the corresponding Young diagram. 
Our interpretation is based on counting certain factorizations of a given
permutation. 
\end{abstract}

\maketitle

% \tableofcontents

\section{Introduction}

\subsection{Generalized Young diagrams}
\label{subsec:generalized}
We are interested in the asymptotics of irreducible representations of the
symmetric groups $\Sym(n)$ for $n\to\infty$ in the scaling of \emph{balanced
Young diagrams} which means that we consider a sequence $(\lambda^{(n)})$ of
Young diagrams  with a property that $\lambda^{(n)}$ has $n$ boxes and
$O(\sqrt{n})$ rows and columns. This scaling makes the graphical representations
of Young diagrams particularly useful; in this article we will use two
conventions for drawing Young diagrams: the French (presented on Figure
\ref{fig:french}) and the Russian one (presented on Figure \ref{fig:russian}).
Notice that the graphs in the Russian convention are created from the graphs in
the French convention by rotating counterclockwise by $\frac{\pi}{4}$ and by
scaling by a factor $\sqrt{2}$.

Any Young diagram drawn in the French convention can be identified with its
graph which is equal to the set
$\{ (x,y): 0\leq x, 0\leq y\leq f(x) \} $ for a suitably chosen function
$f:\R_+\rightarrow\R_+$, where $\R_+=[0,\infty)$. It is therefore natural to
define the set of \emph{generalized Young diagrams
$\Y$ (in the French convention)} as the set of bounded, non-increasing functions
$f:\R_+\rightarrow\R_+$ with a compact support; in this way any Young diagram
can be regarded as a generalized Young diagram.

\begin{figure}[t]
\includegraphics{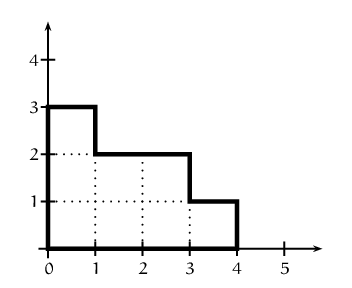} 
\caption{Young diagram $(4,3,1)$ drawn in the French convention}
\label{fig:french}
\end{figure}

We can identify a Young diagram drawn in the Russian convention with its
profile, see Figure \ref{fig:russian}. It is therefore natural to define the set
of \emph{generalized Young diagrams $\Y$ (in the Russian convention)} as the set
of functions $f:\R\rightarrow\R_+$ which fulfill the following two conditions:
\begin{itemize}
 \item $f$ is a Lipschitz function with constant $1$, i.e.\ $|f(x)-f(y)|\leq
|x-y|$, 
\item $f(x)=|x|$ if $|x|$ is large
enough. 
\end{itemize}

\begin{figure}[t]
\includegraphics{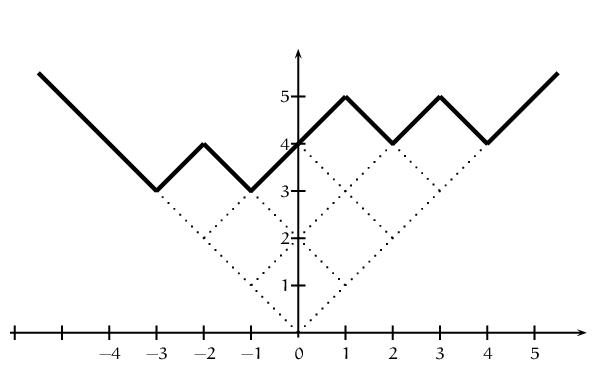} 
\caption{Young diagram $(4,3,1)$ drawn in the Russian convention. The profile
of the diagram has been drawn in the solid line.}
\label{fig:russian}
\end{figure}

At the first sight it might seem that we have defined the set $\Y$ of
generalized Young diagrams in two different ways, but we prefer to think that
these two definitions are just two conventions (French and Russian) for drawing
the same object. This will not lead to confusions since it will be always clear
from the context which of the two conventions is being used.

The setup of generalized Young diagrams makes it possible to speak about
dilations of Young diagrams. In the geometric language of French and Russian
conventions such dilations simply correspond to dilations of the graph.
Formally speaking, if $f\in\Y$ is a generalized Young
diagram (no matter in which convention) and $s>0$ is a real number we define
the dilated diagram $sf\in\Y$ by the formula
$$ (sf)(x)= s f\left(\frac{x}{s}\right). $$
This notion of dilation is very useful in the study of balanced Young diagrams
because if $(\lambda^{(n)})_n$ is a sequence of balanced Young diagrams we may
for example ask questions about the limit of the sequence $\frac{1}{\sqrt{n}}
\lambda^{(n)}$ \cite{LoganShepp1977,VersikKerov1977}.

\subsection{Normalized characters}
Any permutation $\pi\in \Sym(k)$ can be also regarded as an element of $\Sym(n)$
if $k\leq n$ (we just declare that $\pi\in \Sym(n)$ has additional $n-k$
fixpoints).
For any $\pi\in \Sym(k)$ and an irreducible representation $\rho^\lambda$ of the
symmetric group $\Sym(n)$ corresponding to the Young diagram $\lambda$ we
define the normalized character
$$ \Sigma^\lambda_\pi =\begin{cases} \underbrace{n(n-1)\cdots (n-k+1)}_{k \text{
factors}} \frac{\Tr \rho^\lambda(\pi)}{\text{dimension of $\rho^\lambda$}} &
\text{if } k\leq n, \\ 0 & \text{otherwise.} \end{cases}
$$
One of the reasons why such normalized characters are so useful in
the asymptotic representation theory is that, as we shall see in Section
\ref{sec:stanley}, one can extend the definition of
$\Sigma^\lambda_\pi$ to the case when $\lambda\in\Y$ is a generalized Young
diagram; furthermore computing their values will turn out to be easy.
% Our goal in this article is to find explicit formulas for normalized
% characters.

Particularly interesting are the values of characters on cycles, therefore we
will use the notation
$$ \Sigma^\lambda_k = \Sigma^\lambda_{(1,2,\dots,k)}, $$
where we treat the cycle $(1,2,\dots,k)$ as an element of $\Sym(k)$ for any
integer
$k\geq 1$.

\subsection{Free cumulants}
Let $\lambda$ be a (generalized) Young diagram. We define its \emph{free
cumulants} $R_2^\lambda,R_3^\lambda,\dots$ by the formula
\begin{equation} 
\label{eq:free-cumulants-strange}
R_{k}^\lambda = \lim_{s\to\infty} \frac{1}{s^{k}}
\Sigma^{s\lambda}_{k-1}, 
\end{equation}
in other words each free cumulant is asymptotically the dominant term of the
character on a cycle of appropriate length in the limit when the Young diagram
tends 
% in some sense 
to infinity.

From the above definition it is clear that free cumulants should be interesting
for investigations of the asymptotics of characters of symmetric groups, but
it is not obvious why the limit should exist and if there is some direct
way of calculating it. In Sections \ref{sec:functionals-of-measures} and
\ref{sec:generalized-Young-diagrams} we will review some more conventional
definitions of free cumulants and some more direct ways of calculating them.

One of the reasons why free cumulants are so useful in the asymptotic
representation theory is that they are homogeneous with respect to dilations of
the Young diagrams, namely
$$ R^{s\lambda}_k = s^k R^{\lambda}_k;$$
in other words the degree of the free cumulant $R_k$ is equal to $k$. This
property is an immediate consequence of \eqref{eq:free-cumulants-strange}
but 
% we shall see that 
it also follows from more conservative definitions of free cumulants.

In fact, the notion of free cumulants origins from the work of Voiculescu
\cite{Voiculescu1986} where they appeared as coefficients of an $R$-series
which turned out to be useful in description of \emph{free convolution} in the
context of \emph{free probability theory} \cite{VoiculescuDykemaNica1992}. The
name of free cumulants was coined by Speicher \cite{Speicher1998} who found
their combinatorial interpretation and their relations with the lattice of
non-crossing partitions \cite{Speicher1993}. Since free probability theory is
closely related to the random matrix theory \cite{Voiculescu1991} free
cumulants quickly became an important tool not only within the framework of
free probability but in the random matrix theory as well.

\subsection{Kerov character polynomials}
The following surprising fact is fundamental for this article: it turns out 
that free cumulants can be used not only to provide asymptotic
approximations for the characters of symmetric groups, but also for exact
formulas. Kerov during a talk in Institut Henri Poincar\'e in January 2000
\cite{Kerov2000talk} announced the following result (the first published proof
was given by Biane \cite{Biane2003}): for each permutation $\pi$ there exists a
universal polynomial $K_{\pi}$ with integer coefficients, called \emph{Kerov
character polynomial}, with a property that
\begin{equation} 
\label{eq:kerov-polynomial}
\Sigma^{\lambda}_\pi =
K_{\pi}(R_2^{\lambda},R_3^\lambda,\dots)
\end{equation}
holds true for any (generalized) Young diagram $\lambda$.
We say that Kerov polynomial is universal because it does not depend on the
choice of $\lambda$. In order to keep the notation simple we make the dependence
of the characters and of the free cumulants on $\lambda$ implicit and we write 
$$ \Sigma_\pi = K_{\pi}(R_2,R_3,\dots).$$
As usual, we are mostly concerned with the values of the characters on the
cycles, therefore we introduce special notation for such Kerov polynomials
$$ \Sigma_k = K_{k}(R_2,R_3,\dots).$$
Kerov also found the leading term of the Kerov polynomial:
\begin{equation} 
\label{eq:expansion}
\Sigma_k = R_{k+1} + (\text{terms of degree at most $k-1$})
\end{equation}
which has \eqref{eq:free-cumulants-strange} as an immediate consequence.

The first few Kerov polynomials $K_k$ are as follows \cite{Biane2001a}:
\begin{align*}
\Sigma_1 &= R_2, \\
\Sigma_2 &= R_3, \\
\Sigma_3 &= R_4 + R_2,   \\
\Sigma_4 &= R_5 + 5R_3,    \\     
\Sigma_5 &= R_6 + 15R_4 + 5R_2^2 + 8R_2, \\
\Sigma_6 &= R_7 + 35R_5 + 35R_3 R_2 + 84R_3.
\end{align*}
Based on such numerical evidence Kerov formulated during his talk
\cite{Kerov2000talk} the following conjecture.
\begin{conjecture}[Kerov]
\label{conj:Kerov}
The coefficients of Kerov character polynomials $K_{k}$ ($k\geq 1$) are
\emph{non-negative} integers.
% , therefore they have some combinatorial interpretation.
\end{conjecture}
Biane \cite{Biane2003} stated a very interesting conjecture that the underlying
reason for positivity of the coefficients of Kerov polynomials is that 
they are equal to cardinalities of some combinatorial objects. Biane provided
also some heuristics what these combinatorial objects could be (we postpone the
details until Section \ref{subsubsec:biane}).

Since then a number of partial answers were found. Śniady \cite{'Sniady2006a}
found explicitly the next term (with degree $k-1$) in the expansion
\eqref{eq:expansion} (the form of this next term was conjectured by
Biane \cite{Biane2003}). Goulden and Rattan \cite{GouldenRattan2007} found an
explicit but complicated formula for the coefficients of Kerov polynomials.
These results, however, did not shed too much light into possible combinatorial
interpretations of Kerov character polynomials.

Some light on the possible combinatorial interpretation of Kerov polynomials
was shed by the following result proved by Biane in the
aforementioned paper \cite{Biane2003} and Stanley \cite{Stanley2002}.
\begin{theorem}[Linear terms of Kerov polynomials]
\label{theo:linear}
For all integers $l\geq 2$ and $k\geq 1$ the coefficient of $R_l$ in the Kerov
polynomial $K_{k}$ is equal to the number of pairs $(\sigma_1,\sigma_2)$ of
permutations $\sigma_1,\sigma_2\in\Sym(k)$ such that $\sigma_1 \circ
\sigma_2=(1,2,\dots,k)$ and such that $\sigma_2$ consists of one cycle and
$\sigma_1$ consists of\/ $l-1$ cycles.
\end{theorem}

For a permutation $\pi$ we denote by $C(\pi)$ the set of cycles of $\pi$.
Féray \cite{F'eray-preprint2008} extended the above result to the quadratic
terms of Kerov polynomials.
\begin{theorem}[Quadratic terms of Kerov polynomials]
\label{theo:quadratic}
For all integers $l_1,l_2 \geq 2$ and $k\geq 1$ the coefficient of $R_{l_1}
R_{l_2}$ in
the Kerov polynomial $K_{k}$ is equal to the number of triples
$(\sigma_1,\sigma_2,q)$ with the following properties:
% \begin{enumerate}[label=(\alph*)]
\begin{itemize}
 \item 
% \label{enum:quadratic-a} 
$\sigma_1,\sigma_2$ is a factorization of the
cycle; in other words $\sigma_1,\sigma_2\in \Sym(k)$ are such that $\sigma_1
\circ
\sigma_2=(1,2,\dots,k)$;
 \item $\sigma_2$ consists of two cycles and $\sigma_1$ consists of $l_1+l_2-2$
cycles;
 \item 
% \label{enum:quadratic-c} 
$q:C(\sigma_2)\rightarrow \{l_1,l_2\}$ is a
surjective map on the
two cycles of $\sigma_2$;
 \item for each cycle $c\in C(\sigma_2)$ there are at least $q(c)$ cycles of
$\sigma_1$ which intersect nontrivially $c$.
\end{itemize}
\end{theorem}
In fact, Féray \cite{F'eray-preprint2008} managed also to prove positivity of
the coefficients of Kerov character polynomials by finding some combinatorial
objects with appropriate cardinality, but his proof was so complicated that the
resulting combinatorial objects were hardly explicit in more complex cases. We 
compare this work with our new result in Section \ref{sect:compar}

% The following result was conjectured by Biane \cite{Biane2003}.
% \begin{theorem}[Śniady \cite{'Sniady2006a}] For $k\geq 1$ the expansion of 
% $K_k$ into the terms of the two highest degrees is  given by
% \begin{multline*} \Sigma_{k}= R_{k+1} + \\ \frac{(k-1)k(k+1)}{24}
% \sum_{\substack{j_2,j_3,\dots\geq 0, \\  2 j_2+3 j_3+\cdots=k-1}}
% \frac{(j_2+j_3+\cdots)!}{j_2! j_3! \cdots} \prod_{i\geq 2} \big( (i-1)
% R_i)^{j_i}+ \\ \text{(terms of degree at most $k-3$)}.
% \end{multline*}
% \end{theorem}

\subsection{The main result: explicit combinatorial interpretation of the
coefficients of Kerov polynomials}

The following theorem is the main result of the paper: it gives a
satisfactory answer for the Kerov conjecture by providing an explicit
combinatorial interpretation of the coefficients of the Kerov polynomials.
It was formulated for the first time as a conjecture in June 2008 by Valentin
Féray and Piotr Śniady after some computer experiments concerning the
coefficient of $R_2^3$ in Kerov polynomials $K_7$ and $K_9$. The original
formulation of the conjecture was Theorem \ref{theo:main-reformulated}; the
form below was pointed out by Philippe Biane in a private communication.

\begin{theorem}[The main result]
\label{theo:main}
Let $k\geq 1$ and let $s_2,s_3,\dots$ be a sequence of non-negative integers
with only finitely
many non-zero elements. The coefficient of $R_2^{s_2} R_3^{s_3} \cdots $ in
the Kerov polynomial $K_{k}$ is equal to the number of triples
$(\sigma_1,\sigma_2,q)$ with the following properties:
\begin{enumerate}[label=(\alph*)]
 \item \label{enum:first-condition}
$\sigma_1,\sigma_2$ is a factorization of the
cycle; in other words $\sigma_1,\sigma_2\in \Sym(k)$ are such that $\sigma_1
\circ \sigma_2=(1,2,\dots,k)$;
 \item \label{enum:ilosc2} the number of cycles of $\sigma_2$ is equal to the
number of\/ factors in
the product $R_2^{s_2} R_3^{s_3} \cdots $; in other words
$|C(\sigma_2)|=s_2+s_3+\cdots$;
 \item \label{enum:boys-and-girls} the total number of cycles of $\sigma_1$ and
$\sigma_2$ is equal to the degree of the product
$R_2^{s_2} R_3^{s_3} \cdots $; in other words
$|C(\sigma_1)|+|C(\sigma_2)|=2 s_2+3 s_3+4 s_4+\cdots$;
 \item \label{enum:kolorowanie} $q:C(\sigma_2)\rightarrow \{2,3,\dots\}$ is a
coloring of the cycles of
$\sigma_2$ with a property that each color $i\in\{2,3,\dots\}$ is used exactly
$s_i$ times (informally, we can think that $q$ is a map which to cycles of
$C(\sigma_2)$ associates the factors in the product $R_2^{s_2} R_3^{s_3}
\cdots$);
% we require that for every
% color $i\in\{2,3,\dots\}$ there are exactly $s_i$ cycles of $\sigma_2$
% with color $i$;
 \item \label{enum:marriage} for every set $A\subset C(\sigma_2)$
which is
nontrivial (i.e., $A\neq\emptyset$ and $A\neq C(\sigma_2)$) there are more than
$\sum_{i\in A} \big( q(i)-1 \big) $ cycles of\/ $\sigma_1$ which intersect
$\bigcup A$.
\end{enumerate}
\end{theorem}

A careful reader may notice that condition \ref{enum:ilosc2} in the above
theorem is redundant since it is implied by condition \ref{enum:kolorowanie};
we decided to keep it for the sake of clarity. We postpone presenting 
interpretations of condition \ref{enum:marriage} until Section
\ref{subsec:condition-e} and Section \ref{subsec:condition-e-transportation}.

One can easily see that Theorem \ref{theo:linear} and Theorem
\ref{theo:quadratic} are special cases of the above result.
We decided to postpone the discussion of other applications of this main result
until Section \ref{subsec:applications} when more context will be available.

\subsection{Characters for more complicated conjugacy classes}
In order to study characters on more complicated conjugacy classes we will use
the following notation. For $k_1,\dots,k_l\geq 1$ we define
$$ \Sigma_{k_1,\dots,k_l}^\lambda= \Sigma^\lambda_{\pi},$$
where $\pi\in \Sym(k_1+\cdots+k_l)$ is any permutation with the lengths of the
cycles given by $k_1,\dots,k_l$; we may take for example
$\pi=(1,2,\dots,k_1)(k_1+1,k_1+2,\dots,k_1+k_2)\cdots$. For simplicity we will
often suppress the explicit dependence of $\Sigma_{k_1,\dots,k_l}$ on $\lambda$.

Unfortunately, as it was pointed out by Rattan and Śniady
\cite{Rattan'Sniady2008}, Kerov conjecture is not true for more
complicated Kerov polynomials $K_{\pi}$ for which $\pi$ consists of more than
one cycle. However, they conjectured that it would still hold true if the
definition \eqref{eq:kerov-polynomial} of Kerov polynomials was modified as
follows. 

For $k_1,\dots,k_l\geq 1$ we consider \emph{cumulant
$\kappa^{\id}(\Sigma_{k_1},\dots,\Sigma_{k_l})$ of the conjugacy classes of
cycles}. Precise definition of these quantities can be found in
\cite{'Sniady2006c}, for the purpose of this article it is enough to know that
their relation to the characters $\Sigma_{k_1,\dots,k_l}$ is analogous to the
relation between classical cumulants of random variables and their moments, as
it can be seen on the following examples:
\begin{align*}
\Sigma_{r}     = & \kappa^{\id}(\Sigma_r),\\ 
\Sigma_{r,s}   = & \kappa^{\id}(\Sigma_r,\Sigma_s) +
                      \kappa^{\id}(\Sigma_r)\ \kappa^{\id}(\Sigma_s), \\ 
\Sigma_{r,s,t} = &  \kappa^{\id}(\Sigma_r, \Sigma_s, \Sigma_t) +
\kappa^{\id}(\Sigma_r) \kappa^{\id}(\Sigma_s, \Sigma_t) +
\kappa^{\id}(\Sigma_s) \kappa^{\id}(\Sigma_r, \Sigma_t) + \\ &
\kappa^{\id}(\Sigma_t) \kappa^{\id}(\Sigma_r, \Sigma_s) +
\kappa^{\id}(\Sigma_r) \kappa^{\id}(\Sigma_s) \kappa^{\id}(\Sigma_s),
\end{align*}
\begin{align*}
\kappa^{\id}(\Sigma_r) & = \Sigma_{r},\\ 
\kappa^{\id}(\Sigma_r,\Sigma_s) & = \Sigma_{r,s}-\Sigma_r\Sigma_s, \\ 
\kappa^{\id}(\Sigma_r, \Sigma_s, \Sigma_t)&  = \Sigma_{r,s,t} - \Sigma_r
\Sigma_{s,t} - \Sigma_s\Sigma_{r,t} - \Sigma_t \Sigma_{r,s} + 2\, \Sigma_r
\Sigma_s \Sigma_t.
\end{align*}
As it was pointed out in \cite{'Sniady2006c}, the above quantities
$\kappa^{\id}(\Sigma_r,\Sigma_s,\dots)$ are very useful in the study of
fluctuations of random Young diagrams; in fact they are even more fundamental
than the characters $\Sigma_{r,s,\dots}$ themselves.

\begin{conjecture}[Rattan, Śniady \cite{Rattan'Sniady2008}]
\label{conj:rattan-sniady}
For $k_1,\dots,k_l\geq 1$ there exists a universal polynomial
$K_{k_1,\dots,k_l}$ with non-negative integer coefficients, called
\emph{generalized Kerov polynomial}, such that
$$ (-1)^{l-1} \kappa^{\id}(\Sigma_{k_1},\dots,\Sigma_{k_l}) =
K_{k_1,\dots,k_l}(R_2,R_3,\dots). $$
The coefficients of this polynomials have some combinatorial interpretation.
\end{conjecture}

The existence of such a universal polynomial with integer coefficients follows
directly from the work of Kerov. The positivity of the coefficients was proved
by Féray \cite{F'eray-preprint2008} but his combinatorial interpretation of
the coefficients was not very explicit.

In this article will will also prove the following generalization of
Theorem \ref{theo:main} which gives an explicit combinatorial solution to
Conjecture \ref{conj:rattan-sniady}.
\begin{theorem}
\label{theo:rattan-sniady-true}
Let $k_1,\dots,k_l\geq 1$ and let $s_1,s_2,\dots$ be a sequence of non-negative
integers with only finitely many non-zero elements. The coefficient of
$R_2^{s_2} R_3^{s_3} \cdots $ in the generalized Kerov polynomial
$K_{k_1,\dots,k_l}$ is equal to the number of triples
$(\sigma_1,\sigma_2,q)$ which fulfill the same conditions as in Theorem
\ref{theo:main} with the following modification: condition
\ref{enum:first-condition} should be replaced by the following one:
\begin{enumerate}[label=(\alph*')]
 \item $\sigma_1,\sigma_2\in \Sym(k_1+\cdots+k_l)$ are such that
$$\sigma_1\circ \sigma_2=(1,2,\dots,k_1)(k_1+1,k_1+2,\dots,k_1+k_2)\cdots$$ and
the group $\langle \sigma_1,\sigma_2\rangle$ acts transitively on the set
$\{1,\dots,k_1+\cdots+k_l\}$.
\end{enumerate}

\end{theorem}

\subsection{Idea of the proof: Stanley polynomials}
The main idea of the proof of the main result (Theorem \ref{theo:main} and
Theorem \ref{theo:rattan-sniady-true}) is to use Stanley polynomials which are
defined as follows. For two finite sequences of positive real numbers
$\p=(p_1,\dots,p_m)$
and $\q=(q_1,\dots,q_m)$ with $q_1\geq \cdots \geq q_m$ we consider a
multirectangular generalized Young diagram $\p\times \q$,
cf~Figure \ref{fig:multirectangular}. In the case when
$p_1,\dots,p_m,q_1,\dots,q_m$ are natural numbers $\p\times\q$ is a partition
% \begin{equation} 
% \label{eq:multirectangular}
$$\mathbf{p}\times \mathbf{q} = (\underbrace{q_1,\dots,q_1}_{p_1 \text{ times}},
\underbrace{q_2,\dots,q_2}_{p_2 \text{ times}},\dots).  $$ 
% \end{equation}

\begin{figure}[t]
\includegraphics{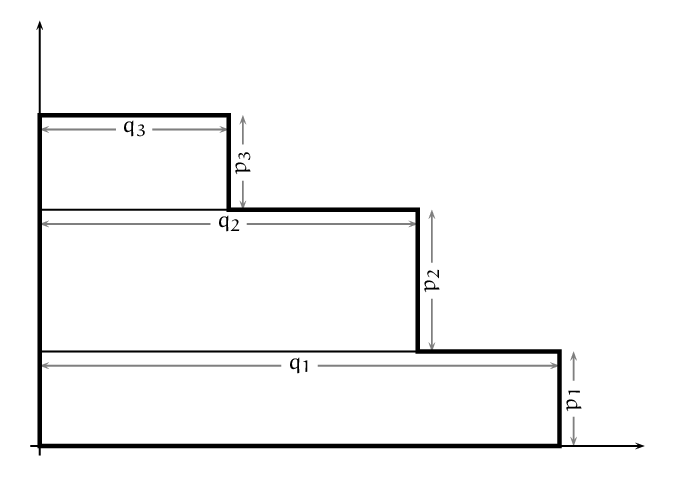}
\caption{Generalized Young diagram $\p\times \q$ drawn in the French convention}
\label{fig:multirectangular}
\end{figure}

If $\F:\Y\rightarrow\R$ is a sufficiently nice function on the set of
generalized Young diagrams (in this article we use the class of, so called,
\emph{polynomial functions}) then $\F(\p\times\q)$ turns out to be a polynomial
in indeterminates $p_1,p_2,\dots,q_1,q_2,\dots$ which will be called
\emph{Stanley polynomial}. The Stanley polynomial for the most interesting
functions $\F$, namely for the normalized characters $\Sigma_{\pi}$, is provided
by Stanley-Féray character formula (Theorem \ref{theo:stanley-feray-old})
which was conjectured by Stanley \cite{Stanley-preprint2006} and proved by
Féray \cite{F'eray-preprint2006}, for a more elementary proof we refer to
\cite{F'eray'Sniady-preprint2007}.

In the past analysis of some special coefficients of Stanley polynomials
resulted in partial results concerning Kerov polynomials
\cite{Stanley2002,Stanley2003/04}. In Theorem \ref{theo:stanley-and-S} we will
show that, in fact, a large class coefficients of Stanley polynomials can be
interpreted as coefficients
$$ \left. \frac{\partial}{\partial S_{k_1}} \cdots
\frac{\partial}{\partial S_{k_l}} \F \right|_{S_2=S_3=\cdots=0} $$
in the Taylor expansion of $\F$ into the \emph{basic functionals
$S_2,S_3,\dots$ of shape} of a Young diagram. 

These basic functionals $S_2,S_3,\dots$ of shape are not new; they already
appeared (possibly with a slightly modified normalization) in the work of
Ivanov and Olshanski \cite{IvanovOlshanski2002} and implicitly in the work of
Kerov \cite{Kerov1998,Kerov1999,Kerov2003}.

In this way we are able to express $\Sigma_{\pi}$ as an explicit polynomial in
$S_2,S_3,\dots$. In Proposition \ref{prop:relation-between-r-and-s} we will
show how to express $S_2,S_3,\dots$ in terms of free cumulants $R_2,R_3,\dots$.
Finally, we use some identities fulfilled by Stanley polynomials (Lemma
\ref{lem:tozsamosci-wielomianow-stanleya}) in order to express the coefficients
of Kerov polynomials in a useful way.

% \section{Alternative interpretations of condition \ref{enum:marriage}}
% \label{sec:interpretations}

\subsection{Combinatorial interpretation of condition \ref{enum:marriage}} 
\label{subsec:condition-e}
Let $(\sigma_1,\sigma_2,q)$ be a triple which fulfills conditions
\ref{enum:first-condition}--\ref{enum:kolorowanie} of Theorem \ref{theo:main}.
We consider the following polyandrous interpretation of Hall marriage theorem.
Each cycle of $\sigma_1$ will be called a \emph{boy} and each cycle of
$\sigma_2$ will be called a \emph{girl}. For each girl $j\in C(\sigma_2)$ let
$q(j)-1$ be the \emph{desired number of husbands of $j$} (notice that condition
\ref{enum:boys-and-girls} shows that the number of boys in $C(\sigma_1)$ is
right so that if no other restrictions were imposed it would be possible to
arrange marriages in such a way that each boy is married to exactly one girl
and each girl has the desired number of husbands). We say that a boy $i\in
C(\sigma_1)$ is a possible candidate for a husband for a girl $j\in C(\sigma_2)$
if cycles $i$ and $j$ intersect. Hall marriage theorem applied to our setup says
that there exists an arrangement of marriages $\M:C(\sigma_1)\rightarrow
C(\sigma_2)$ which assigns to each boy his wife (so that each girl $j$ has
exactly $q(j)-1$ husbands) if and only if 
% \begin{enumerate}[label=(e')]
% \item \label{enum:marriage-prim} 
for every set $A\subseteq C(\sigma_2)$ there are
at least $\sum_{i\in A} \big( q(i)-1 \big) $ cycles of\/ $\sigma_1$ which
intersect $\bigcup A$.
% \end{enumerate}
As one easily see, the above condition is similar but not identical to
\ref{enum:marriage}.

\begin{proposition}
Condition \ref{enum:marriage} is equivalent
to the following one:
\begin{enumerate}[label=(\ee)]
\item \label{enum:marriage-prim} 
for every nontrivial set of girls $A\subset C(\sigma_2)$ (i.e.,
$A\neq\emptyset$ and $A\neq C(\sigma_2)$) there exist two ways of arranging
marriages $\M_p:C(\sigma_1)\rightarrow C(\sigma_2)$, $p\in\{1,2\}$ for
which the corresponding sets of husbands of wives from $A$ are different:
$$ \M_1^{-1}(A)\neq \M_2^{-1}(A). $$
\end{enumerate}
\end{proposition}
\begin{proof}
The implication \ref{enum:marriage-prim}$\implies$\ref{enum:marriage} is
immediate.

For the opposite implication Hall marriage theorem shows existence of $\M_1$.
Let us select any boy $i\in \M_1^{-1}(A)$ and let us declare that boy $i$ is 
not allowed to marry any girl from the set $A$. Applying Hall marriage theorem
for the second time shows existence of $\M_2$ with the required properties
which finishes the proof of equivalence.
\end{proof}

\label{subsec:defgraphs}
For permutations $\sigma_1,\sigma_2$ it is convenient to introduce a bipartite
graph $\V^{\sigma_1,\sigma_2}$ with the set of vertices $C(\sigma_1)\sqcup
C(\sigma_2)$ with edges connecting intersecting cycles
\cite{F'eray'Sniady-preprint2007}. The elements of $C(\sigma_1)$, respectively
$C(\sigma_2)$, will be referred to as white, respectively black, vertices. For
a bipartite graph with a vertex set $V$ we will denote by $V_\bullet$ the set
of black vertices. 
% ; for each $l \in [1;k]$, we draw an
% edge between vertices $i\in C(\sigma_1)$ and $j\in C(\sigma_2)$ corresponding
% to the cycles containing $l$: it is a multi-edge version of the graph
% introduced in \cite{F'eray'Sniady-preprint2007}. 

The following
result gives a strong restriction on the form of the factorizations which
contribute to Theorem \ref{theo:main} and we hope it will be useful in the
future investigations of Kerov polynomials. Notice that this kind of result
appears also in the work of Féray \cite{F'eray-preprint2008}.
\begin{proposition}
\label{prop:restriction-on-topology}
Suppose that $\sigma_1,\sigma_2\in\Sym(k)$ are such that in the graph
$\V^{\sigma_1,\sigma_2}$ there exists a disconnecting edge $e$ with a property
that each of the two connected components of the resulting truncated graph
$\V^{\sigma_1,\sigma_2}\setminus\{e\}$ contains at least one vertex from
$C(\sigma_2)$. Then triple $(\sigma_1,\sigma_2,q)$ never contributes to the
quantities described in Theorem \ref{theo:main} and Theorem
\ref{theo:rattan-sniady-true}, no matter how $q$ and $s_2,s_3,\dots$ are chosen.
\end{proposition}
\begin{proof}
Before starting the proof notice that the 
assumptions of Theorem \ref{theo:main} and Theorem
\ref{theo:rattan-sniady-true} show that in order for $(\sigma_1,\sigma_2,q)$ to
contribute, graph $\V^{\sigma_1,\sigma_2}$ must be connected.

Let $i\in C(\sigma_1)$, $j\in C(\sigma_2)$ be the endpoints of the edge $e$
and let $V_1,V_2$ be the connected components of
$\V^{\sigma_1,\sigma_2}\setminus\{e\}$; we may assume that $i\in V_1$ and $j\in
V_2$.
We are going to use the condition \ref{enum:marriage-prim}. Let $A$,
respectively $B$, be the set of girls, respectively boys, contained in $V_1$.
From the assumption it follows that $V_1\neq\{i\}$ therefore
$A\neq\emptyset$; on the other hand $j\in V_2$ therefore $A\neq
C(\sigma_2)$. 

If 
\begin{equation} 
\label{eq:roznica-chlopcy-dziewczynki}
|B|-\sum_{j\in A} \big( q(j)-1 \big)  
\end{equation}
does not belong to the set $\{0,1 \}$ then it is not possible to arrange the
marriages.

If \eqref{eq:roznica-chlopcy-dziewczynki} is equal to zero then any
arrangement of marriages $\M:C(\sigma_1)\rightarrow C(\sigma_2)$ must fulfill
$\M^{-1}(A)=B$; if \eqref{eq:roznica-chlopcy-dziewczynki} is equal to $1$
% the number of boys in $V_1$
% is equal to the disired number of husbands of the girls from $V_1$ 
then any
arrangement of marriages $\M:C(\sigma_1)\rightarrow C(\sigma_2)$ must fulfill
$\M^{-1}(A)=B\setminus\{i\}$. In both cases, the set of husbands of
wives from $A$ is uniquely determined therefore condition
\ref{enum:marriage-prim} does not hold.
\end{proof}

\subsection{Transportation interpretation of condition \ref{enum:marriage}} 
\label{subsec:condition-e-transportation}

Let $G$ be a bipartite graph and its set of black vertices be $V_\bullet$. For
any set $A\subseteq V_\bullet$ of black vertices we denote by $N_G(A)$ the set
of white vertices which have a neighbor in $A$. We can rephrase condition
\ref{enum:marriage} by:
\begin{enumerate}[label=(\eee)]
\item \label{enum:marriage-graphs} for any non-trivial subset $A$, $\quad
\left| N_{\V^{\sigma_1,\sigma_2}}(A) \right| \geq 1+\sum_{c \in A} \big[q(c)
-1\big]$.
\end{enumerate}

Let a coloring $q:V_\bullet\rightarrow\{2,3,\dots\}$ of the black vertices of 
a bipartite graph $G$ be given. We say that $G$ is $q$-admissible if for every
set $A\subseteq V_\bullet$ of black vertices $\left| N_A\right| \geq \sum_{c\in
A} \big[ q(c)-1
\big] $ and furthermore the equality holds if and only if $V_\bullet$ is equal
to the set of all black vertices in a union of some connected components of $G$.

Notice that if $G$ is connected then it is $q$-admissible if and only if it
satisfies condition \ref{enum:marriage-graphs}. 
% Also, if $G$ contains no loops
% then it is $q$-admissible if and only if it is a collection of trees of the
% following form: one black vertex $c\in V_\bullet$ connected to $q(c)-1$
% leaves.

\begin{proposition}
\label{prop:existence-of-solution}
Condition \ref{enum:marriage} is equivalent to the following one:
\begin{enumerate}[label=(\eeee)]
\item \label{enum:transportation}
there exists a strictly positive solution to the following system of equations: 
\begin{description}
 \item[Set of variables]\ \\ $\big\{x_{i,j}: \text{ white vertex $i$ is
connected to black vertex $j$} \big\}$
 \item[Equations] $\left\{\begin{array}{l} \forall i, \sum_j x_{i,j}=1 \\
\forall j, \sum_i x_{i,j}=q(j)-1 \end{array}\right.$
\end{description}
\end{enumerate}

More generally, graph $G$ is $q$-admissible if and only if condition
\ref{enum:transportation} is fulfilled.
\end{proposition}
% Let $q:V_\bullet\rightarrow\{2,3,\dots\}$ be given. We say that a bipartite
% graph $G$ with the set of black vertices $V_\bullet$ is admissible if for
% every
% set $A\subseteq V_\bullet$ of black vertices $N_A\geq \sum_{c\in A} \big[
% q(c)-1
% \big] $ and furthermore the equality holds if and only if $V_\bullet$ is equal
% to the set of all black vertices in a union of some connected components.
Before starting the proof note that the possibility of arranging marriages
(see Section \ref{subsec:condition-e}) can be rephrased as existence of a
solution to the above system of equations with a requirement that
$x_{i,j}\in\{0,1\}$. 

The system of equations in condition
\ref{enum:transportation} can be interpreted as a transportation problem where
each white vertex is interpreted as a factory which produces a unit of some ware
and each black vertex $j$ is interpreted as a consumer with a demand equal to
$q(j)-1$. The value of $x_{i,j}$ is interpreted as amount of ware transported
from factory $i$ to the consumer $j$.

\begin{proof}
Suppose that the above system has a positive solution. For any $A\subseteq
V_\bullet$ we have
$$\sum_{j \in A} (q(j) -1) = \sum_{j \in A} \sum_{\substack{i:\\ 
(i,j)\text{ is an edge}}} x_{i,j} = \sum_{i \in N_G(A)} \sum_{\substack{j \in
A:\\(i,j)
\text{ is an edge}}} x_{i,j} \leq |N_G(A)|.$$
Furthermore, if $|N_G(A)| =\sum_{j \in A} (q(j)-1)$ then the above inequality is
an
equality which means that for each $i \in N_G(A)$ one has
$$\sum_{\substack{j \in A:\\(i,j) \text{ is an edge}}} x_{i,j}=1.$$
As $\sum\limits_j x_{i,j} =1$ and $x_{i,j}>0$ if $(i,j)$ is an edge, this
implies that there is no edge $(i,j)$ with $i \in N_G(A)$ and $j \notin A$. In
this way we have proved that $A$ is the set of black vertices of a union of some
disjoint components, therefore $G$ is $q$-admissible.

The opposite implication is easy: we consider the mean of all solutions of the
system with the condition $x_{i,j} \in \{0,1\}$. This gives us a strictly
positive solution because the $q$-admissibility ensures that if we
force some variable $x_{i,j}$ to be equal to $1$ we can find a solution to the
system.
\end{proof}

% \section{Outlook}
% \label{sec:outlook}

\subsection{Open problems}

\subsubsection{$C$-expansion}

% \begin{multline*} \Sigma_{k}= R_{k+1} + \\ \frac{(k-1)k(k+1)}{24}
% \sum_{\substack{j_2,j_3,\dots\geq 0, \\  2 j_2+3 j_3+\cdots=k-1}}
% \frac{(j_2+j_3+\cdots)!}{j_2! j_3! \cdots} \prod_{i\geq 2} \big( (i-1)
% R_i)^{j_i}+ \\ \text{(terms of degree at most $k-3$)}.
% \end{multline*}
% 
% 

In analogy to \eqref{eq:free-cumulants-strange} we define for $k\geq 2$
\begin{equation}
\label{eq:definicja-ck-dziwna}
 C_k^{\lambda} =\frac{24}{k(k+1)(k+2)} \lim_{s\to\infty}
\frac{1}{s^k} 
\left( \Sigma_{k+1}^{s \lambda}- R_{k+2}^{s\lambda} \right) 
\end{equation}
which (up to the unusual numerical factor in front) gives the leading terms of
the deviation from the first-order approximation 
$\Sigma_{k+1}^\lambda \approx R_{k+2}^\lambda$.
The explicit form of $C_k$ 
$$
C_k= 
% \frac{k(k+1)(k+2)}{24}
\sum_{\substack{j_2,j_3,\dots\geq 0, \\  2 j_2+3 j_3+\cdots=k}}
\frac{(j_2+j_3+\cdots)!}{j_2! j_3! \cdots} \prod_{i\geq 2} \big( (i-1)
R_i)^{j_i}
$$
as a polynomial in free cumulants $R_2,R_3,\dots$
was conjectured by Biane \cite{Biane2003} and was proved by Śniady
\cite{'Sniady2006a}. Goulden and Rattan \cite{GouldenRattan2007} proved that
for each $k\geq 1$ there exists a universal polynomial $L_k$ called
\emph{Goulden--Rattan polynomial} with rational coefficients such that
\begin{equation} 
\label{eq:rozwiniecie-c}
\Sigma_{k}-R_{k+1} = L_k(C_2,C_3,\dots) 
\end{equation}
and they found an explicit but complicated formula for $L_k$. A simpler
proof and some more related results can be found in the work of Biane
\cite{Biane2005/07}.
\begin{conjecture}[Goulden and Rattan \cite{GouldenRattan2007}]
\label{conj:goulden-rattan}
The coefficients of $L_k$ are \emph{non-negative} rational numbers with
relatively small denominators.
\end{conjecture}
It is natural to conjecture that the underlying reason for positivity of the 
coefficients is that they have (after some additional rescaling)  a
combinatorial interpretation.

In some sense the free cumulants $(R_k)$ are analogous to the above quantities
$(C_k)$:  both have natural interpretations as leading (respectively,
subleading) terms in the asymptotics of characters,
cf.~\eqref{eq:free-cumulants-strange}, respectively
\eqref{eq:definicja-ck-dziwna}. Also, Conjecture \ref{conj:Kerov} is analogous
to Conjecture \ref{conj:goulden-rattan}: both conjectures state that there are
exact formulas which express the characters $\Sigma_k$ (respectively, the
subdominant terms of the characters $\Sigma_k-R_{k+1}$) as polynomials in free
cumulants (respectively, $(C_k)_{k\geq 2}$) with non-negative integer
coefficients (respectively, non-negative rational coefficients with small
denominators) which have a combinatorial interpretation.

The advantage of the quantities $(C_k)$ over free cumulants $(R_k)$ is that the
Goulden--Rattan polynomials $L_k$ seem to have a simpler form than Kerov
polynomials $K_k$ while the numerical evidence for Conjecture
\ref{conj:goulden-rattan} suggests that their coefficients should have a rich
and beautiful structure. Also, Kerov's conjecture (Conjecture \ref{conj:Kerov})
would be an immediate corollary from Conjecture \ref{conj:goulden-rattan}.
For these reasons we tend to believe that the quantities $(C_k)$ are even
better suitable for the asymptotic representation theory then the free
cumulants $(R_k)$ and Conjecture \ref{conj:goulden-rattan} deserves
serious interest.

\subsubsection{$\mathcal R$-expansion}
Another interesting direction of research was pointed out by Lassalle
\cite{Lassalle-preprint2007} who presented quite explicit conjectures on the
form of the coefficients of Kerov polynomials.

\subsubsection{Arithmetic properties of Kerov polynomials}

\begin{proposition}
If $p$ is an odd prime number then $ \frac{\Sigma_{p}-R_{p+1}+2 R_2}{p} $ and 
$\frac{\Sigma_{p-1}-R_p}{p}$ are polynomials in free cumulants $R_2,R_3,\dots$
with nonnegative integer coefficients.
\end{proposition}
\begin{proof}
In order to prove that the coefficients of $ \frac{\Sigma_{p}-R_{p+1}+2 R_2}{p}$
are integer we consider the action of the group $\mathbb{Z}/p\mathbb{Z}$ on the
set of triples $(\sigma_1,\sigma_2,q)$ which contribute to Theorem
\ref{theo:main} defined by conjugation $$\psi(i) (\sigma_1,\sigma_2,q) = \big(
c^i \sigma_1 c^{-i}, c^i \sigma_2 c^{-i}, q'),$$ where $c=(1,2,\dots,k)$ is the
cycle; we leave the details how to define $q'$ as a simple exercise. All orbits
of this action consist of $p$ elements except for the fixpoints of this action
which are of the form $\sigma_1=c^a$, $\sigma_2=c^{1-a}$. These fixpoints
contribute to the monomial $R_{k+1}$ (with multiplicity $1$) and to the monomial
$R_2$ (with multiplicity $p-2$). 

In order to prove that the coefficients of $ \frac{\Sigma_{p-1}-R_{p}}{p}$
are integer we express $R_{p}$ as a linear combination of the
conjugacy classes $\Sigma_{\pi}$. A formula for such an expansion presented in
the paper \cite{'Sniady2006a} involves summation over all partitions of the set
$\{1,\dots,p\}$. The group $\mathbb{Z}/p\mathbb{Z}$ acts on such partitions; all
orbits in this action consist of $p$ elements except for the fixpoints of this
action: the minimal partition (which gives $\Sigma_{p-1}$) and the maximal
partition (which turns out not to contribute). We express all summands (except
for the summand corresponding to $\Sigma_{p-1}$) as polynomials in free
cumulants, which finishes the proof.
\end{proof}

The following conjecture was formulated by Światosław Gal (private
communication) based on numerical calculations.
\begin{conjecture}
If $p$ is an odd prime number then
 and
$\frac{\Sigma_{p+1}-R_{p+2}+R_3}{p}$
is a polynomial in free cumulants $R_2,R_3,\dots$ with nonnegative integer
coefficients.
\end{conjecture}

We hope that the above claims will shed some light on more precise structure of
Kerov polynomials and on the form of $C$-expansion and $\mathcal{R}$-expansion
described above; they suggest that (maybe up to some small error terms)
$ \Sigma_n - R_{n+1}$ should in some sense be divisible by $(n-1)n(n+1)$ which
supports the conjectures of Lassalle \cite{Lassalle-preprint2007}.

\subsubsection{Discrete version of the functionals $S_2,S_3,\dots$}
\label{subsec:discrete-version}
One of the fundamental ideas in this paper is the use of the fundamental
functionals $S_2,S_3,\dots$ of the shape of a Young diagram defined as
integrals over the area of a Young diagram of the powers of the contents:
\begin{equation}
\label{eq:pierwsze-pojawienie-S}
S^\lambda_n = (n-1) \iint_{\Box\in\lambda} 
(\contents_\Box)^{n-2}\ d\Box 
\end{equation}
(we postpone the precise definition until Section \ref{subsec:functionals}).

It would be interesting to investigate properties of analogous quantities
\begin{equation}
\label{eq:T}
T^\lambda_n = (n-1) \sum_{\Box\in\lambda} (\contents_\Box)^{n-2}
\end{equation}
in which the integral over the Young diagram was replaced by a sum over its
boxes. Notice that unlike the integrals \eqref{eq:pierwsze-pojawienie-S} which
are well-defined for generalized Young diagrams, the sum \eqref{eq:T} makes
sense only if $\lambda$ is a conventional Young diagram but since the resulting
object is a polynomial function on the set of Young diagrams it can be extended
to generalized Young diagrams. This type of quantities have been investigated
by Corteel, Goupil and Schaeffer \cite{CorteelGoupilSchaeffer2004}.

The reason why we find the functional $T_n$ so interesting is that via
non-commutative Fourier transform it corresponds to a central element of the
symmetric group algebra $\C[\Sym(k)]$ given by the following very simple formula
$$T_n = (n-1) \sum_{2\leq i\leq n} X_i^{n-2}, $$
where 
$$X_i= (1,i)+(2,i)+\cdots+(i-1,i)  \in\C[\Sym(k)]$$
are the Jucys-Murphy elements. 

The hidden underlying idea behind the current paper is the differential
calculus on the (polynomial) functions on the set of generalized Young diagrams
$\Y$ in which we study derivatives corresponding to infinitesimal changes of
the shape of a Young diagram, as it can be seen in the proof of Theorem
\ref{theo:stanley-and-S}. It is possible to develop the formalism of such a
differential calculus and to express the results of this paper in such a
language instead of the language of Stanley polynomials (and, in fact, the
initial version of this article was formulated in this way), nevertheless if the
main goal is to prove the Kerov conjecture then this would lead to unnecessary
complication of the paper.

On the other hand, just like the usual differential and integral calculus has
an interesting discrete difference and sum analogue, the above described
differential calculus on generalized Young diagrams has a discrete difference
analogue in which we study the change of the function on the set of Young
diagrams corresponding to addition or removal of a single box. We expect that
just like functionals $(S_n)$ are so useful in the framework of differential
calculus on the set of generalized Young diagrams, functionals $(T_n)$ will be
useful in the framework of the difference calculus on Young diagrams.

It would be very interesting to develop such a difference calculus and to
verify if free cumulants $(R_n)$ have some interesting discrete version which
nicely fits into this setup.

\subsubsection{Characterization of Stanley polynomials}

Lemma \ref{lem:tozsamosci-wielomianow-stanleya} contains some identities
fulfilled by Stanley polynomials. It would be interesting to find some more
such identities. In particular we state the following problem here.

\begin{problem}
Find (minimal set of) conditions which fully characterize the class of Stanley
polynomials $\F(\p\times\q)$ where $\F$ is a polynomial function on the set of
Young diagrams.
\end{problem}

It seems plausible that the answer for this problem is best formulated in the
language of the differential calculus of function on the set of generalized
Young diagrams about which we mentioned in Section
\ref{subsec:discrete-version}.

\subsubsection{Various open problems}

Is there some analogue of Kerov character polynomials for the representation
theory of semisimple Lie groups, in particular for the unitary groups $U(d)$? 
Does existence of Kerov polynomials for characters of symmetric groups
$\Sym(n)$ tell us something (for example via Schur-Weyl duality) about
representations of the unitary groups $U(d)$? Is there some analogue of Kerov
character polynomials in the random matrix theory? Is it possible to study Kerov
polynomials in such a scaling that phenomena of universality of random matrices
occur?

\subsection{Exotic interpretations of Kerov polynomials} 
Theorem \ref{theo:main} gives some interpretation of the coefficients of Kerov
polynomials but clearly it does not mean that there are no other
interpretations. 

\subsubsection{Biane's decomposition} 
\label{subsubsec:biane}
The original conjecture of Biane \cite{Biane2003} suggested that the
coefficients of Kerov polynomials are equal to multiplicities in some
unspecified decomposition of the Cayley graph of the symmetric group into a
signed sum of non-crossing partitions. This result was proved by Féray
\cite{F'eray-preprint2008} but the details of his construction were quite
implicit. In Section \ref{sect:compar} we shall revisit the conjecture of Biane
in the light of our new combinatorial interpretation of Kerov polynomials.
Unfortunately, our understanding of this interpretation of the coefficients of
Kerov polynomials is still not satisfactory and remains as an open problem.

\subsubsection{Multirectangular random matrices}
For a given Young diagram $\lambda$ we consider a Gaussian random matrix
$(A^\lambda_{ij})$ with the shape of $\lambda$. Formally speaking, the entries
of $(A^\lambda_{ij})$ are independent with $A^\lambda_{ij}=0$ if box $(i,j)$
does not belong to $\lambda$; otherwise $\Re A^\lambda_{ij},\Im A^\lambda_{ij}$
are independent Gaussian random variables with mean zero and variance
$\frac{1}{2}$. One can think that either $(A^\lambda_{ij})$ is an infinite
matrix or it is a square (or rectangular) matrix of sufficiently big size.

\begin{theorem}
Kerov polynomials express the moments of the random matrix $A^\lambda$ in terms
of the genus-zero terms in the genus expansion (up to the sign). More
precisely, 
$$ \mathbb{E} \left[ \Tr \big( A^\lambda (A^\lambda{})^{\star}\big)^n \right] =
-K_n(-R_2,-R_3,-R_4,\dots), $$
where $R_i$ is defined as the genus zero term in the expansion for 
$$\mathbb{E} \left[ \Tr \left( A^\lambda
\left(A^\lambda\right)^{\star}\right)^{i-1} \right],$$
or, precisely speaking, 
$$ R_i= \lim_{s\to\infty}  \frac{1}{s^{i}} \mathbb{E} \left[ \Tr \left(
A^{s\lambda}
\left(A^{s\lambda}\right)^{\star}\right)^{i-1} \right].$$
\end{theorem}

This is an immediate consequence of the results from
\cite{F'eray'Sniady-preprint2007}.

% \subsubsection{Okounkov's plantation}
% We conjecture that the coefficient $[C_2^{s_2} C_3^{s_3} \cdots] \Sigma_n$ in
% the $C$-expansion of the character is (asymptotically, can we make some
% precise
% statement here?) equal to the number of $2$-partitions of a $2n$-elements set
% which in the Okounkov's decompositions \cite{Okounkov2000} give rise to $s_2$
% plantations of length $2$, $s_3$ plantations of length $3$,\dots.
% 
% The name plantation does not appear really in Okounkov's paper, but I do not
% know how to call this object. Okounkov decomposes a glued together polygon
% into
% a bunch of straight lines---called plantations--- to which collections of
% trees
% are attached.

\subsubsection{Dimensions of (co)homologies}
In analogy to Kazhdan-Lusztig polynomials it is tempting to ask
if the coefficients of Kerov polynomials might have a topological
interpretation, for example as dimensions of (co)ho\-mo\-lo\-gies of some
interesting geometric objects, maybe related to Schubert varieties, as suggested
by Biane (private communication). This would be supported by the Biane's
decomposition from Section \ref{subsubsec:biane} which maybe is related to
Bruhat order and Schubert cells. In this context it is interesting to ask if the
conditions from Theorem \ref{theo:main} can be interpreted as geometric
conditions on intersections of some geometric objects. Another approach towards
establishing link between Kerov polynomials and Schubert calculus would be to 
relate Kerov polynomials and Schur symmetric polynomials.

%  \item Is the decomposition of the Cayley graph suggested by Biane and found
% by
% Féray in fact a decomposition of a smaller graph, given by Bruhat order? If
% yes, this would simplify a lot looking for connections with Schubert calculus.
% Is there some nice property which makes this decomposition unique or almost
% unique? A walk on the Cayley graph should correspond to a walk on Schubert
% cells.
% \item Can we compare Kerov polynomials with Schubert polynomials and
% Grothendick polynomials?
% \item If we look for occurences of non-crossing partitions in the Schubert
% calculus: maybe one should have a look at \emph{vexillary permutations} or
% \emph{grassmannian permutations}?
% \item 
% \item Suggested reading: Laurent Manivel, \emph{Symmetric Functions, Schubert
% Polynomials and Degeneracy Loci}.

\subsubsection{Schur polynomials}
Each Schur polynomial can be written as quotient of two determinants.
Exactly the same quotient of determinants appears in the
Harish-Chandra-Itzykson-Zuber integral
$$ \int_{U(d)}   e^{A U B U^\star } dU $$
if $A$ and $B$ are hermitian matrices with suitably chosen eigenvalues
(say $(x_i)$ for $A$ and $(\log \lambda_i)$ for $B$).

It would be interesting to verify if Kerov polynomials can be
used to express the exact values of
Schur polynomials by some limit value of Harish-Chandra-Itzykson-Zuber integral
when the size of the matrix tends to infinity and each variable $x_i$
occurs with a multiplicity which tends to infinity; also the shape of the Young
diagram $\lambda$ should tend to infinity, probably in the ``balanced Young
diagram'' way.

\subsubsection{Analytic maps}
We conjecture that Kerov polynomials are related to moduli space of analytic
maps on Riemann surfaces or ramified coverings of a sphere.

\subsubsection{Integrable hierarchy} 
Jonathan Novak (private communication) conjectured that Kerov polynomials
might be algebraic solutions to some integrable
hierarchy (maybe Toda?) and their coefficients are related to the tau function
of the hierarchy.

\subsection{Applications of the main result}
\label{subsec:applications}

\subsubsection{Positivity conjectures and precise information on Kerov
polynomials}

The advantage of the approach to characters of symmetric groups presented in
this article over some other methods is that the formula for the coefficients
given by Theorem \ref{theo:main} does not involve summation of terms of
positive and negative sign unlike most formulas for characters such as
Murnaghan-Nakayama rule or Stanley-Féray formula (Theorem
\ref{theo:stanley-feray-old}). In this way we avoid dealing with complicated
cancellations. For this reason the main result of the current paper seems to be
a
perfect tool for proving stronger results, such as the Conjecture
\ref{conj:goulden-rattan} of Goulden and Rattan or the conjectures of Lassalle
\cite{Lassalle-preprint2007}.

\subsubsection{Genus expansion} One of the important methods in the random
matrix theory and in the representation theory is to express the quantity we
are interested in (for example: moment of a random matrix or character of a
representation) as a sum indexed by some combinatorial objects (for example:
partitions of an ordered set or maps) to which one can associate canonically a
two-dimensional surface \cite{LandoZvonkin2004}. Usually the asymptotic
contribution of such a summand depends on the topology of the surface with
planar objects being asymptotically dominant. This method is called genus
expansion since exponent describing the rate of decay of a given term usually
linearly depends on the genus.

The main result of this article fits perfectly into this philosophy since to any
pair of permutations $\sigma_1$, $\sigma_2$ which contributes to Theorem
\ref{theo:main} or Theorem \ref{theo:rattan-sniady-true} we may associate a
canonical graph on a surface, called a \emph{map}. It is not difficult to show
that also in this situation the degree of the terms $R_2^{s_2} R_3^{s_3}\cdots$ 
to which such a pair of permutations contributes decreases as the genus
increases.

It is natural therefore to ask about the structure of factorizations $\sigma_1
\circ \sigma_2=(1,2,\dots,k)$ with a prescribed genus. As we already pointed
out in Proposition \ref{prop:restriction-on-topology}, condition
\ref{enum:marriage} of Theorem \ref{theo:main} gives strong limitations on the
shape of the resulting bipartite graph $\V^{\sigma_1,\sigma_2}$ which
translate to limitations on the shape of the corresponding map. Very analogous
situation was analyzed in the paper \cite{'Sniady2006a} where it was proved that
by combining a restriction on the genus and a condition analogous to the one
from Proposition \ref{prop:restriction-on-topology} (\emph{``evercrossing
partitions''}) one gets only a finite number of allowed patterns for the
geometric object concerned.

Similar analysis should be possible for the formulas for Kerov polynomials
presented in the current paper which should shed some light on Conjecture
\ref{conj:goulden-rattan} of Goulden and Rattan and the conjectures of Lassalle
\cite{Lassalle-preprint2007}.

\subsubsection{Upper bounds on characters}
It seems plausible that the main result of this article, Theorem \ref{theo:main}
and Theorem \ref{theo:rattan-sniady-true}, can be used to prove new upper bounds
on the characters of symmetric groups
\begin{equation} 
\label{eq:charaktery}
\chi^\lambda(\pi)= \frac{\Tr \rho^\lambda(\pi)}{\text{dimension
of $\rho^\lambda$}} 
\end{equation}
for balanced Young diagram $\lambda$ in the scaling when the length of the
permutation $\pi$ is large compared to the number of boxes of $\lambda$.

The advantage of such approach to estimates on characters over other methods,
such as via Frobenius formula as in the work of Rattan and Śniady
\cite{Rattan'Sniady2008} or via Stanley-Féray formula
\cite{F'eray'Sniady-preprint2007}, becomes particularly visible in the case when
the shape of the Young diagram becomes close to the limit curve for the
Plancherel measure \cite{LoganShepp1977,VersikKerov1977} for which all free
cumulants (except for $R_2$) are close to zero. Indeed, for $\lambda$ in the
neighborhood of this limit curve one should expect much tighter bounds on the
characters \eqref{eq:charaktery} because such Young diagrams maximize the
dimension of the representation which is the denominator of the fraction, while
the numerator can be estimated by Murnaghan-Nakayama rule and some combinatorial
tricks \cite{Roichman1996}.

\subsection{Overview of the paper}
% 
% Sections \ref{sec:interpretations} and \ref{sec:outlook} are intended as an
% extension of the introduction: in Section \ref{sec:interpretations} we
% present alternative interpretations of condition \ref{enum:marriage} in
% Theorem \ref{theo:main} and in Section \ref{sec:outlook} we discuss open
% problems and possible applications of the main result.

In Section \ref{sec:functionals-of-measures} we recall some basic facts about
free cumulants $R_1,R_2,\dots$ and quantities $S_1,S_2,\dots$ for probability
measures on the real line and their relations with each other. The main
result of this section is formula \eqref{eq:s-r} which allows to express
functionals $S_1,S_2,\dots$ in terms of free cumulants $R_1,R_2,\dots$.

In Section \ref{sec:generalized-Young-diagrams} we define the fundamental
functionals $S_2,S_3,\dots$ for generalized Young diagrams and study their
geometric interpretation.

In Section \ref{sec:stanley} we study Stanley polynomials and their relations to
the fundamental functionals $S_2,S_3,\dots$ of shape of a Young diagram.

Section \ref{sec:toyexample} is devoted to a toy example: we shall prove
Theorem \ref{theo:main} in the simplest non-trivial case of
coefficients of the quadratic terms, which is exactly the case in Theorem
\ref{theo:quadratic}. In this way the Reader can see all essential steps of the
proof in a simplified situation when it is possible to avoid technical
difficulties.

In Section \ref{sec:combinatorial-lemmas} we prove some auxiliary combinatorial
results.

In Section \ref{sec:proof} we present the proof of the main result:
Theorem \ref{theo:main} and Theorem \ref{theo:rattan-sniady-true}.

Finally, in Section \ref{sect:compar} we revisit the paper
\cite{F'eray-preprint2008} and we show how rather implicit constructions
of Féray become much more concrete once one knows the formulation of the main
result of the current paper, Theorem \ref{theo:main}. In fact, Section
\ref{sect:compar} provides a short proof of Theorem \ref{theo:main} based on
results from of Féray. This simplicity is however slightly misleading since the
original paper \cite{F'eray-preprint2008} is not easy.

\section{Functionals of measures}
\label{sec:functionals-of-measures}
In this section we present relations between moments $M_1,M_2,\dots$ of a
given probability measure, its free cumulants $R_1,R_2,\dots$ and its
functionals $S_1,S_2,\dots$. The only result of this section which will be used
in the remaining part of the article is equality \eqref{eq:s-r}, nevertheless we
find functionals $S_1,S_2,\dots$ so important that we collected in this section
also some other formulas involving them.

Assume that $\nu$ is a compactly supported measure on $\R$.
% Assume that $\nu$ is a compactly supported Schwartz
% distribution on $\R$. In order to keep notation simple, if
% $\phi:\R\rightarrow\R$ is a sufficiently smooth function we will write $$ \int
% \phi(z) d\nu(z):= \langle \nu, \phi\rangle, $$ i.e.~in the notation we will
% treat distributions in a similar way as measures on $\R$. 
For integer $n\geq 0$ we consider moments of $\nu$ 
$$ M^{\nu}_n= \int z^n d\nu(z) $$ 
and its Cauchy transform $$ G^{\nu}(z)=
\int \frac{1}{z-x} d\nu(x) =\sum_{n\geq 0} \frac{M_n^{\nu}}{z^{1+n}};$$ 
the integral and the series make sense in a neighborhood of infinity.

% In case when 
% $$\int 1\ d\nu(z)=1$$ 
% (for example, if $\nu$ is a probability measure) 
From the following on we assume that $\nu$ is a compactly supported
probability measure on $\R$.
We define a sequence $(S^\nu_n)_{n\geq 1}$ of the coefficients of the expansion
$$ S^{\nu}(z)=\log z G^{\nu}(z) = \sum_{n\geq 1} \frac{S^{\nu}_n}{z^n} $$ 
in a neighborhood of infinity
and a sequence $(R^{\nu}_n)_{n\geq 1}$  of free cumulants as the coefficients of
the expansion
\begin{equation}
\label{eq:free-cumulants}  
R^{\nu}(z)=\left( G^{\nu} \right)^{\langle -1\rangle}(z) - \frac{1}{z}=
\sum_{n\geq 1}
R^{\nu}_n z^{n-1}  
\end{equation}
in a neighborhood of $0$, where $\left( G^{\nu} \right)^{\langle -1\rangle}$ is
the right inverse of $G^{\nu}$ with respect to the composition of functions
\cite{Voiculescu1986}.
When it does not lead to confusions we shall omit the superscript in the
expressions $M^{\nu}_n$, $G^{\nu}$, $S^{\nu}$, $R^{\nu}$, $S^{\nu}_n$,
$R^{\nu}_n$.

The relation between the moments and the free cumulants is given by
the following combinatorial formula which, in fact, can be regarded
as an alternative definition of free cumulants \cite{Speicher1998}:
\begin{equation}
\label{eq:moment-cumulant}
M_n = \sum_{\Pi\in\NC_n } R_{\Pi},
\end{equation}
where the summation is carried over all non-crossing partitions of $n$-element
set and where $R_{\Pi}$ is defined as the multiplicative extension of $(R_k)$:
$$ R_{\Pi} = \prod_{b\in \Pi} R_{|b|}, $$
where the product is taken over all blocks $b$ of the partition $\Pi$ and $|b|$
denotes the number of the elements in $b$ \cite{Speicher1998}.

Information about the measure $\nu$ can be described in various ways; in this
article descriptions in terms of the sequences $(S_n)$ and $(R_n)$ play eminent
role and we need to be able to relate each of these sequences to the other.
We shall do it in the following.

\begin{lemma}
\label{lem:pochodnapowolnych}
For any integer $k\geq 1$
$$ \frac{\partial G(R_1,R_2,\dots)}{\partial R_k}(z) = - \frac{1}{k}\bigg(
\big[G(z)\big]^k \bigg)'= - G^{k-1}(z) G'(z), $$
where both sides of the above equality are regarded as formal power series in
powers of $\frac{1}{z}$ with the 
coefficients being polynomials in $R_1,R_2,\dots$. 
\end{lemma}
\begin{proof}
Equation \eqref{eq:free-cumulants} is equivalent to
\begin{equation}
\label{eq:inverse}
 G\left( R(z)+\frac{1}{z} \right) = z. 
\end{equation}
We denote 
$$ t= R(z)+\frac{1}{z}. $$
Let us keep all free cumulants fixed except for $R_k$, we shall treat $G$ as a
function of free cumulants. By taking the derivatives of both sides of
\eqref{eq:inverse} it follows that
\begin{multline*}
0=\frac{\partial}{\partial R_k}\left[ G\left(
R(z)+\frac{1}{z}
\right)\right] = \frac{\partial G}{\partial R_k} (t) + G'(t)
\frac{\partial}{\partial R_k}\left(
R(z)+\frac{1}{z}
\right) = \\
\frac{\partial G}{\partial R_k} (t) + G'(t) z^{k-1} =
\frac{\partial G}{\partial R_k} (t) + G'(t)\cdot  G^{k-1}(t) 
\end{multline*}
which finishes the proof. 
\end{proof}

\begin{proposition}
\label{prop:relation-between-r-and-s}
% Sequences $(S_n)$ and $(R_n)$ corresponding to the same probability measure
% are related to each other by equalities
For any integer $n\geq 1$
\begin{align}
\label{eq:momentykumulanty}
 M_n &= \sum_{l\geq 1} \frac{1}{l!}  (n)_{l-1} 
\sum_{\substack{k_1,\dots,k_l\geq 1 \\ k_1+\cdots+k_l=n }} R_{k_1} \cdots
R_{k_l},\\
\label{eq:s-r}
 S_n &= \sum_{l\geq 1} \frac{1}{l!}  (n-1)_{l-1} 
\sum_{\substack{k_1,\dots,k_l\geq 1 \\ k_1+\cdots+k_l=n }} R_{k_1} \cdots
R_{k_l}, \\
\label{eq:r-s}
R_{n} &=
 \sum_{l\geq 1} \frac{1}{l!} (-n+1)^{l-1} \sum_{\substack{k_1,\dots,k_l\geq 1 \\
k_1+\cdots+k_l=n }} S_{k_1} \cdots S_{k_l},
\end{align}
where $$(a)_b=\underbrace{a (a-1) \cdots (a-b+1)}_{b \text{ factors}}$$ denotes
the falling factorial.
\end{proposition}

\begin{proof}
Lemma \ref{lem:pochodnapowolnych} shows that
$$ \frac{\partial^2 G(R_1,R_2,\dots)}{\partial R_k \ \partial R_l}(z) =  
\frac{1}{k+l-1} \bigg(\big[G(z)\big]^{k+l-1} \bigg)'' $$
therefore if $k+l=k'+l'$ then
$$ \frac{\partial^2 M_n(R_1,R_2,\dots)}{\partial R_k \ \partial R_l} =
\frac{\partial^2 M_n(R_1,R_2,\dots)}{\partial R_{k'} \ \partial R_{l'}}.$$
It follows by induction that 
$$ \frac{\partial^l M_n(R_1,R_2,\dots)}{\partial R_{k_1} \cdots \partial
R_{k_l}}=
\frac{\partial^l M_n(R_1,R_2,\dots)}{(\partial R_{1})^{l-1} \partial
R_{k_1+\cdots+k_l-(l-1)}}.
$$
From the moment-cumulant formula \eqref{eq:moment-cumulant} it follows that for
$R_1=R_2=\cdots=0$ the
right-hand side of the above equation is equal to the number of non-crossing
partitions with an ordering of blocks, such that the numbers of elements in
consecutive blocks are as follows:  
$$\underbrace{1,\dots,1}_{l-1 \text{ times}},k_1+\cdots+k_l-(l-1).$$ 
Such non-crossing partitions have a particularly simple structure therefore it
is very easy to find their cardinality. Therefore
\begin{equation} 
\label{eq:pochodne-z-m}
\left. \frac{\partial^l M_n(R_1,R_2,\dots)}{(\partial
R_{1})^{l-1} \partial
R_{k_1+\cdots+k_l-(l-1)}} 
\right|_{R_1=R_2=\cdots=0} = 
\begin{cases}
(n)_{l-1} & \text{if } n=k_1+\cdots+k_l, \\
0 & \text{otherwise,}
\end{cases}
\end{equation}
which finishes the proof of \eqref{eq:momentykumulanty}. 

Lemma \ref{lem:pochodnapowolnych} shows that for $k\geq 2$
$$ \frac{\partial S(R_1,R_2,\dots)}{\partial R_k}(z)= \frac{\partial \log \big[
z G(z)\big]}{\partial R_k}= 
-G^{k-2} G' = \frac{\partial G(R_1,R_2,\dots)}{\partial R_{k-1}} $$
therefore 
$$ \frac{\partial S_n(R_1,R_2,\dots)}{\partial R_k}=\frac{\partial
M_{n-1}(R_1,R_2,\dots)}{\partial R_{k-1}}.$$
Assume that $k_l\geq 2$; then
$$\frac{\partial^l S_n(R_1,R_2,\dots)}{\partial R_{k_1} \cdots \partial
R_{k_l}} 
=\frac{\partial^l M_{n-1}(R_1,R_2,\dots)}{\partial R_{k_1} \cdots \partial
R_{k_l-1}} $$
% The right hand side of the above equality for $R_1=R_2=\cdots=0$ is equal to
% $$ \left. \frac{\partial^l M_{n-1}(R_1,R_2,\dots)}{\partial R_{k_1} \cdots
% \partial R_{k_l-1}} \right|_{R_1=R_2=\cdots=0} =
% \begin{cases}
% (n-1)_{l-1} & \text{if } n=k_1+\cdots+k_l, \\
% 0 & \text{otherwise.}
% \end{cases}
% $$
which is calculated in Eq.~\eqref{eq:pochodne-z-m}. In this way we proved
that if $(k_1,\dots,k_l)\neq (1,1,\cdots,1)$ then 
$$ \left. \frac{\partial^l S_n(R_1,R_2,\dots)}{\partial R_{k_1} \cdots \partial
R_{k_l}} \right|_{R_1=R_2=\cdots=0} =
\begin{cases}
(n-1)_{l-1} & \text{if } n=k_1+\cdots+k_l, \\
0 & \text{otherwise.}
\end{cases}$$
In order to prove the case $k_1=\cdots=k_l=1$ it is enough to consider the
Dirac point measure $\nu=\delta_a$ for which $G(z)=\frac{1}{z-a}$, $R_1=a$,
$R_2=R_3=\cdots=0$ and $S(z)=-\log \left(1-\frac{a}{z} \right)$,
$S_n=\frac{a^n}{n}$. In this way the proof of \eqref{eq:s-r} is finished.

Lagrange inversion formula shows that
\begin{multline*} 
R_{n+1}= - \frac{1}{n} \left[\frac{1}{z}\right] \left( \frac{1}{G(z)} \right)^n
=
 - \frac{1}{n} \left[\frac{1}{z^{n+1}}\right] \exp [-n S(z)] = \\
 \sum_{l\geq 1} \frac{1}{l!} (-n)^{l-1} \sum_{\substack{k_1,\dots,k_l\geq 1 \\
k_1+\cdots+k_l=n+1 }} S_{k_1} \cdots S_{k_l}
 \end{multline*}
which finishes the proof of \eqref{eq:r-s}.
\end{proof}

\section{Generalized Young diagrams}
\label{sec:generalized-Young-diagrams}

The main result of this section is the formula \eqref{eq:definicja-s} which
relates the fundamental functionals $S_2,S_3,\dots$ to the geometric shape of
the Young diagram. 

In the following we base on the notations introduced in Section
\ref{subsec:generalized}.

\subsection{Measure on a diagram and contents of a box} 

Notice that each unit box of a Young diagram drawn in the French convention
becomes in the Russian notation a square of side $\sqrt{2}$. For this reason,
when drawing a Young diagram according to the French convention we will use the
plane equipped with the usual measure (i.e.\ the area of a unit square is equal
to $1$) and when drawing a Young diagram according to the Russian notation we
will use the plane equipped with the usual measure divided by $2$ (i.e.\ the
area of a unit square is equal to $\frac{1}{2}$). In this way a (generalized)
Young diagram has the same area when drawn in the French and in the Russian
convention.

Speaking very informally, the setup of generalized Young diagrams corresponds
to looking at a Young diagram from very far away so that individual boxes
become very small. Therefore by the term \emph{box of a Young diagram $\lambda$}
we will understand simply any point $\Box$ which belongs to $\lambda$. In the
case of the Russian convention this means that $\Box=(x,y)$ fulfills
$$ |x| < y < \lambda(x). $$
We define the contents of the box $\Box=(x,y)$ in the Russian convention by
$\contents_\Box=x$.

In the case of the French convention $\Box=(x,y)$ belongs to a diagram
$\lambda$ if
$$ x > 0 \qquad \text{ and } \qquad 0 < y < \lambda(x) $$ and the contents of
the box $\Box=(x,y)$ is defined by $\contents_\Box=x-y$.

\subsection{Functionals of Young diagrams}
\label{subsec:functionals}
The above definitions of the measure on the plane and of the contents in the
case of French and Russian conventions are compatible with each other,
therefore it is possible to define some quantities in a convention-independent
way. In particular, we define the  \emph{fundamental functionals of shape} of a
generalized Young diagram
\begin{equation}
\label{eq:definicja-s}
S^\lambda_n = (n-1) \iint_{\Box\in\lambda} 
(\contents_\Box)^{n-2}\ d\Box  
\end{equation}
for integer $n\geq 2$. Clearly, each functional $S_n$ is a homogeneous function
of the Young diagram with degree $n$.

Let a generalized Young diagram $\lambda:\R\rightarrow\R_+$ drawn in the
Russian convention be fixed. We associate to it a function
$$ \tau^{\lambda}(x)=  \frac{\lambda(x)-|x|}{2}$$
% for $x\in\R$ 
which gives the distribution of the contents of the boxes of $\lambda$. When
it does not lead to confusions we will write for simplicity $\tau$ instead of
$\tau^\lambda$. In the following we
shall view $\tau$ as a measure on $\R$. Its Cauchy transform can be written as 
$$ G^\tau(z)  = \iint_{\Box\in\lambda} \frac{1}{z-\contents_\Box} d\Box. $$
With these notations we have that
$$ S^\lambda_n = (n-1) \int x^{n-2}\ \tau(x)\ dx =  - \int x^{n-1}\ \tau'(x)\
dx$$
are (rescaled) moments of the measure $\tau$ or, alternatively,
(shifted) moments of the Schwartz distribution $-\tau'$.

We define
% \begin{multline*} 
$$ S^\lambda(z) = \sum_{n\geq 2} \frac{S^{\lambda}_n}{z^n} =
\iint_{\Box\in\lambda} \frac{1}{(z-\contents_\Box)^2} \ d\Box $$
where the second equality follows by expanding right-hand side into a power
series and \eqref{eq:definicja-s}. It follows that
\begin{multline*} S^\lambda(z)= - \frac{d}{dz}
G^{\tau}(z)=
G^{-\tau'}(z) = -\int \frac{1}{z-x} \tau'(x)\ dx=\\
- \int \log(z-x)\ \tau''(x)\ dx, 
\end{multline*}
% [TO DO: Markov-Krein transform, interlacing measures, Reyleigh measures,
% multiplicative Cauchy transform]
in particular $S^\lambda(z)$ coincides with the Cauchy transform of
a Schwartz distribution $-\tau'$. The above formulas show that $S^\lambda(z)$
and $S_n^\lambda(z)$ coincide (up to small modifications) with the quantities
considered by Kerov
\cite{Kerov1999,Kerov2003}, Ivanov and Olshanski \cite{IvanovOlshanski2002}.

%%%%%%% INFORMACJA:
%%%%%%%  
%%%%%%% \tau_{Kerov} = \tau_{Sniady}''+ \delta_0
%%%%%%%
%%%%%%%

% \begin{multline*}
% S^{\mu_\lambda}(z_0) = - \int \log(z_0-z) \left( \frac{\lambda(z)-|z|}{2}
% \right)'' dz = \\
%  - \int \frac{1}{(z_0-z)^2} \cdot \frac{\lambda(z)-|z|}{2}\  dz =
%  \sum_{k\geq 2} \frac{1}{z_0^i} \int (k-1) z^{k-2} \cdot
% \frac{\lambda(z)-|z|}{2} dz, 
% \end{multline*}
% in other words 
% \begin{multline*} 
% S^{\mu_\lambda})_k =  \int (k-1) z^{k-2} \cdot \frac{\lambda(z)-|z|}{2} dz =
% \\
% \int_{\Box\in\lambda} (k-1) C_\Box^{k-2} d\Box
%  \end{multline*}

\subsection{Kerov transition measure}
The corresponding Cauchy transform 
\begin{equation} 
\label{eq:cauchy-for-young}
G^{\lambda}(z)= \frac{1}{z} \exp S^{\lambda}(z)  
\end{equation}
is a Cauchy transform of a probability measure $\mu_{\lambda}$ on the real
line, called \emph{Kerov transition measure} of $\lambda$
\cite{Kerov1999,Kerov2003}.
Probably it would be more correct to write $G^{\mu_{\lambda}}$ instead of
$G^\lambda$ and to write $S^{\mu_\lambda}$ instead of $S^\lambda$, but this
would lead to unnecessary complexity of the notation.

One of the reasons why Kerov's transition measure was so successful in the
asymptotic representation theory of symmetric groups is that it can be defined
in several equivalent ways, related either to the shape of $\lambda$ or to
representation theory or to moments of Jucys-Murphy elements or to certain
matrices. For a review of these approaches we refer to \cite{Biane1998}.

\subsection{Free cumulants of a Young diagram}
In order to keep the introduction as non-technical as possible, we introduced
free cumulants of a Young diagram by the formula
\eqref{eq:free-cumulants-strange}. The conventional way of defining them is
to use \eqref{eq:free-cumulants} for the Cauchy transform given by
\eqref{eq:cauchy-for-young}. Therefore, one should make sure that these two
definitions are equivalent. This can be done thanks to Frobenius formula
$$ \Sigma_{k-1}^\lambda = - \frac{1}{k-1} \left[\frac{1}{z}\right]
\frac{1}{G^\lambda(z-1) G^\lambda(z-2) \cdots G^\lambda\big(z-(k-1)\big)} $$
which shows that
$$ \frac{1}{s^{k}} \Sigma_{k-1}^{s\lambda} = - \frac{1}{k-1}
\left[\frac{1}{z}\right] \frac{1}{G^\lambda\left(z-\frac{1}{s}\right)
G^\lambda\left(z-\frac{2}{s}\right) \cdots
G^\lambda\left(z-\frac{k-1}{s}\right)}; $$
therefore definition \eqref{eq:free-cumulants-strange} would give
$$ R^\lambda_k =- \frac{1}{k-1} \left[ \frac{1}{z} \right] \left(
\frac{1}{G^\lambda(z)}\right)^{k-1} $$
which coincides with the value given by the Lagrange inversion formula applied
to \eqref{eq:free-cumulants}.

\subsection{Polynomial functions on the set of Young diagrams}

For simplicity we shall often drop the explicit dependence of the functionals
of Young diagrams from $\lambda$.
% denote
% $$ M_n=M_n^{\mu_{\lambda}}, \quad S_n=R_n^{\mu_{\lambda}},\quad R_n =
% R_n^{\mu_{\lambda}}, 
% \quad \Sigma_{\pi}=\Sigma^{\lambda}_{\pi}
% ,
% \quad N_{\sigma_1,\sigma_2}=N^{\lambda}_{\sigma_1,\sigma_2}
% $$
% and we will view them as functions on the set of Young diagrams. 
Since the transition measure $\mu^\lambda$ is always centered it follows that
$M_1=R_1=S_1=0$.

Existence of Kerov polynomials allows us define formally the normalized
characters $\Sigma_{\pi}^\lambda$ even if $\lambda$ is a generalized Young
diagram.

We will say that a function on the set of generalized Young diagrams $\Y$ is a
\emph{polynomial function} if one of the following equivalent conditions hold
\cite{IvanovOlshanski2002}: 
\begin{itemize}
 \item it is a polynomial in $M_2,M_3,\dots$;
 \item it is a polynomial in $S_2,S_3,\dots$;
 \item it is a polynomial in $R_2,R_3,\dots$;
 \item it is a polynomial in $(\Sigma_{\pi})_{\pi}$.
\end{itemize}

\section{Stanley polynomials and Stanley-Féray character formula}
\label{sec:grafsf}
\label{sec:stanley}

\subsection{Stanley polynomials}

% Let $\G=V_1\sqcup V_2$ be a bipartite graph as above. We consider a coloring
% $k:V_2\rightarrow\N$ of the vertices of $V_2$ and we extend it to a coloring
% $k:\G\rightarrow\N$ by declaring that the color $k(v)$ of a vertex $v\in V_1$
% is equal to the maximum of the colors of the adjacent vertices in $V_2$.
% 
% For such a coloring $k$ we define
% $$ \p_k \q_k= 
% \prod_{v\in V_1} p_{k(v)} \prod_{v\in V_2} q_{k(v)}.$$  
% 
% \begin{proposition}
%  $$ N^{\p\times \q}_\G = \sum_k   \p_k \q_k,
% $$
% where the sum runs over all colorings $k$ as described above.
% \end{proposition}

% In the following we shall treat $\p=(p_1,p_2,\dots)$, 
% $\q=(q_1,q_2,\dots)$ as families of indeterminates and we define
% Stanley polynomials by the formula
% $$ N_\G(\p,\q) = \sum_k   \p_k \q_k. $$
% 
% [maybe we shaould rather use a completely different definition of Stanley
% polynomials]

% For given
% permutations $\sigma_1,\sigma_2\in S_l$ we shall consider a bipartite graph
% $\G$
% with the set of vertices $C(\sigma_1)\sqcup C(\sigma_2)$ and with an edge
% between vertices $c_1\in C(\sigma_1)$, $c_2\in C(\sigma_2)$ if $c_1\cap
% c_2\neq
% \emptyset$. In such case we denote
% $N^{\lambda}_{\sigma_1,\sigma_2}=N^{\lambda}_\G$.

% in other words
% \begin{equation}
% \label{eq:warunek}
% 0<h(c_1)\leq \lambda_{h(c_2)}
% \end{equation}
% holds true for all $c_1\in C(\sigma_1)$, $c_2\in C(\sigma_2)$ such that
% $c_1\cap
% c_2\neq\emptyset$.

% The following theorem will be our main tool in the analysis of asymptotics of
% representations
% of symmetric groups.

\begin{proposition}
Let $\F:\Y\rightarrow\R$ be a polynomial function on the set of
generalized Young diagrams. Then $(\p,\q)\mapsto \F(\mathbf
p\times \q)$ for $\p=(p_1,\dots,p_m)$, $\mathbf
q=(q_1,\dots,q_m)$ is a polynomial in indeterminates
$p_1,\dots,p_m,q_1,\dots,q_m$, called \emph{Stanley polynomial}.
% 
% For a permutation $\pi\in S_l$ the corresponding polynomial for the normalized
% character $\Sigma_{\pi}$ is given by
% \begin{equation}
% \label{eq:stanley-feray}
%  \Sigma^{\p\times \q}(\pi)= \text{[to be filled in]}
% \end{equation}
% and for the free cumulant $R_{n+1}$ is given by
% \begin{equation}
% \label{eq:stanley-feray-free}
%  R_{n+1}^{\p \times \q}= \text{[to be filled in]}
% \end{equation}
\end{proposition}
\begin{proof}
It is enough to prove this proposition for some family of generators of the
algebra of polynomial functions on $\Y$ for example for functionals
$S_2,S_3,\dots$. We leave it as an exercise.
\end{proof}

\begin{theorem}
\label{theo:stanley-and-S}
Let $\F:\Y\rightarrow\R$ be a polynomial function on the set of
generalized Young diagrams, we shall view it as a polynomial in $S_2,S_3,\dots$
Then for any $k_1,\dots,k_l\geq 2$
\begin{equation} 
\label{eq:rozwiniecie-w-esy-przez-stanleya}
\left. \frac{\partial}{\partial S_{k_1}} \cdots
\frac{\partial}{\partial S_{k_l}} \F \right|_{S_2=S_3=\cdots=0} = 
[p_1 q_1^{k_1-1} \cdots p_l q_l^{k_l-1}] \F (\p \times \q). 
\end{equation}
\end{theorem}
\begin{proof}
Let $\p=(p_1,\dots,p_m)$, $\q=(q_1,\dots,q_m)$. For a given
index $i$ we consider a trajectory in the set of generalized Young diagrams
$q_i\mapsto (\p\times \q)$, where all other parameters $(p_j)$ and
$(q_j)_{j\neq i}$ are treated as constants. In the Russian convention we have
\begin{multline*} \left( \frac{\partial}{\partial q_i} (\p \times \q)
\right)(x )=\\ 
\begin{cases} 2 & \text{if } q_i-p_1-\cdots-p_i<x<
q_i-p_1-\cdots-p_{i-1},\\ 0 & \text{otherwise},\end{cases}
\end{multline*}
which shows the change of the contents distribution. From
\eqref{eq:definicja-s} it follows therefore
$$ \frac{\partial}{\partial q_i}  S_n^{\p\times
\q}= 
\int_{q_i-p_1-\cdots-p_i}^{q_i-p_1-\cdots-p_{i-1}}\ (n-1) x^{n-2} \ dx $$
and 
$$ \frac{\partial}{\partial q_i}  \F(\p\times \q)= 
\sum_{n\geq 2} \int_{q_i-p_1-\cdots-p_i}^{q_i-p_1-\cdots-p_{i-1}} 
\frac{\partial \F}{\partial S_n}\ (n-1) x^{n-2} \ dx. $$

By iterating the above argument we show that
\begin{multline*} \frac{\partial}{\partial q_1} \cdots \frac{\partial}{\partial
q_l} \F(\p\times \q)= \sum_{n_1,\dots,n_l\geq 2}
\frac{\partial}{\partial S_{n_1}} \cdots \frac{\partial}{\partial S_{n_l}}
\F(\p\times \q)\ \ \times \\
\int_{q_1-p_1}^{q_1} (n_1-1) x_1^{n_1-2}\ dx_1
% \int_{q_2-p_1-p2}^{q_2-p_1} (n_2-1) x_2^{n_1-2}\ dx_2
\cdots
\int_{q_l-p_1-\cdots-p_l}^{q_l-p_1-\cdots-p_{l-1}} 
(n_l-1) x_l^{n_l-2}\ dx_l.
\end{multline*}
We shall treat both sides of the above equality as polynomials in $\p$
and we will treat $\q$ as constants. We are going to compute the
coefficient of $p_1 \cdots p_m$ of both sides; we do this by computing the
dominant term of the right-hand side in the limit $\p\to 0$. It follows
that
\begin{multline*} [p_1 \cdots p_m] 
\frac{\partial}{\partial q_1} \cdots \frac{\partial}{\partial
q_l} \F(\p\times \q)=\\
\sum_{n_1,\dots,n_l\geq 2}
 \left. 
\frac{\partial}{\partial S_{n_1}} \cdots \frac{\partial}{\partial S_{n_l}}
\F(\p\times \q) \right|_{p_1=\cdots=p_l=0} \times \\ 
(n_1-1) q_1^{n_1-2} \cdots (n_l-1) q_l^{n_l-2}
\end{multline*}
which finishes the proof.
\end{proof}

\begin{corollary}
\label{coro:order-does-not-matter}
If $k_1,\dots,k_l\geq 2$ then 
$$ [p_1 q_1^{k_1-1} \cdots p_l q_l^{k_l-1}] \F (\p \times \q)$$ 
does not depend on the order of the elements of the sequence $(k_1,\dots,k_l)$.
\end{corollary}

\subsection{Identities fulfilled by coefficients of Stanley polynomials}
The coefficients of Stanley polynomials of the form $ [p_1 q_1^{k_1-1} \cdots
p_l q_l^{k_l-1}] \F (\p \times \q)$ with $q_1,\dots,q_{k_l}\geq 2$ have a
relatively simple structure, as it can be seen for example in Corollary
\ref{coro:order-does-not-matter}. In the following we will study the properties
of such coefficients if some of the numbers $k_1,\dots,k_l$ are equal to
$1$.

Let $\F:\Y\rightarrow\R$ be a fixed polynomial function. For a sequence
$(a_1,b_1),\dots,(a_m,b_m)$ of ordered pairs, where $a_1,\dots,a_m\geq 2$ and
$b_1,\dots,b_m\geq 1$ are integers we define an auxiliary quantity
\begin{multline*} 
\NN^{\F}_{(a_1,b_1)\dots,(a_m,b_m)}=\\ 
\left( \prod_r (-1)^{b_r-1}\ (a_r-1)_{(b_r-1)} \right)
 [p_1 q_1^{a_1-1} \cdots p_m q_m^{a_m-1}] \F(\p\times \q),
\end{multline*}
which thanks to Corollary \ref{coro:order-does-not-matter} does not depend on
the order of the elements in the tuple $(a_1,b_1),\dots,(a_m,b_m)$.

\begin{corollary} 
\label{coro:jak-policzyc-pochodne-po-R}
For any polynomial function $\F$ on the set of generalized
Young diagrams and $k_1,\dots,k_l\geq 2$
$$ \left. \frac{\partial}{\partial R_{k_1}} \cdots \frac{\partial}{\partial
R_{k_l}} \F\right|_{R_2=R_3=\cdots=0} = 
\sum_{\Pi\in P(1,2,\dots,l)}  (-1)^{l-|\Pi|}\ \NN^{\F}_{\left( \left( \sum_{i\in
b} k_i , |b|\right): b\in \Pi \right)}, $$ 
where the sum runs over all partitions of $\{1,\dots,l\}$.
\end{corollary}
\begin{proof}
It is enough to use Theorem \ref{theo:stanley-and-S} and Equation
\eqref{eq:s-r}.
\end{proof}

\begin{lemma}
\label{lem:tozsamosci-wielomianow-stanleya}
For any polynomial function $\F:\Y\rightarrow\R$ and any sequence of
integers $k_1,\dots,k_m\geq 1$
$$  [p_1 q_1^{k_1-1} \cdots p_m q_m^{k_m-1}] \F(\p\times \q)
= \sum_{\Pi}  \NN^{\F}_{\left( \left( \sum_{i\in b} k_i ,
|b|\right)  : b\in \Pi \right)}, $$ 
where the sum runs over all partitions $\Pi$ of the set $\{1,\dots,m\}$ with a
property that if\/ $(a_1,\dots,a_l)$ with $a_1<\cdots<a_l$ is a block of $\Pi$
then $k_{a_1}=\cdots=k_{a_{l-1}}=1$ and $k_{a_l}\geq 2$ or, in other words, the
set of rightmost legs of the blocks of $\Pi$ coincides with the set of indices
$i$ such that $k_i\geq 2$.
\end{lemma}
\begin{proof}
We shall treat $\F(\p\times \q)$ as a polynomial in
$\p$ and we shall treat $\q$ as constants. Our goal is to
understand the coefficient $[p_1  \cdots p_m ] \F(\p\times
\q)$. Since $\F$ is a polynomial in $S_2,S_3,\dots$ we are also
going to investigate analogous coefficients for $\F=S_n$.

For the purpose of the following calculation we shall use the French notation.
\begin{multline*} S_n(\lambda)= (n-1) \iint (\contents_\Box)^{n-2} \ d\Box =\\
(n-2)! \sum_{1\leq r\leq n-1} (-1)^{r-1} \iint_{(x,y)\in\lambda}
\frac{x^{n-1-r}}{(n-1-r)!} \frac{y^{r-1}}{(r-1)!}\ dx \ dy.
\end{multline*}
Since the integral 
$$ \iint_{(x,y)\in\lambda}
\frac{x^{n-1-r}}{(n-1-r)!} \frac{y^{r-1}}{(r-1)!}\ dx \ dy$$
can be interpreted as the volume of the set
\begin{multline*} \big\{ (x_1,\dots,x_{n-r},    y_1,\dots,y_{r}):
0<x_1<\cdots<x_{n-r} \text{ and } \\ 0<y_1<\cdots<y_{r} \text{ and }
(x_{n-r},y_{r})\in \lambda \big\}
\end{multline*}
therefore for any $i_1<\cdots<i_r$
\begin{equation}
\label{eq:rozwiniecie-s-w-stanleya}
 [p_{i_1} \cdots p_{i_r}] S_n(\p\times \q)= 
% % % % % % % % (n-1)
% % % % % % % % \binom{n-2}{r} r! q_{i_r}^{n-1-r}=
% % % % % % % % 
% % % % % % % % 
(-1)^{r-1}\ (n-1)_{r-1}\ q_{i_r}^{n-r}.
% \begin{cases} 
% (-1)^{r-1}\ (n-1)_{r-1}\ q_{i_r}^{n-r} & \text{for }  1\leq r\leq n-1, \\
% 0      & \text{for } r\geq n.
% \end{cases}
\end{equation}

We express $\F$ as a polynomial in $S_2,S_3,\dots$. Notice that the
monomial $p_1 \dots p_m$ can arise in $\F(\p\times \q)$
only in the following way: we cluster the factors $p_1\cdots p_m$ in all
possible ways or, in other words, we consider all partitions $\Pi$ of the set
$\{1,\dots,m\}$. Each block of such a partition corresponds to one factor $S_n$
for some value of $n$. Thanks to Equation
\eqref{eq:rozwiniecie-s-w-stanleya} we can compare the factors $q_1,\dots,q_m$
which appear with a non-zero exponent and see that only partitions $\Pi$ which
contribute are as prescribed in the formulation of the lemma; furthermore we
can find the correct value of $n$ for each block of $\Pi$.

Equation
\eqref{eq:rozwiniecie-w-esy-przez-stanleya} finishes the proof.
\end{proof}

\subsection{Stanley-Féray character formula}

% We denote by $\Sym(n)^{(r)}$ the set of \emph{colored permutations}, i.e.\
% pairs $(\sigma,\phi)$, where $\sigma\in \Sym(l)$ and
% $\phi:\{1,\dots,n\}\rightarrow
% \{1,\dots,r\}$ is a function constant on each cycle of $\sigma$.
% Alternatively,
% the coloring $\phi$ can be regarded as a function on $C(\sigma)$, 
% the set cycles of $\sigma$.
% Given a colored permutation $(\sigma,\phi)\in \Sym(n)^{(r)}$ and a non-colored
% one $\pi\in \Sym(n)$ we define their product $(\sigma,\phi) \cdot \pi=
% (\sigma \circ \pi, \psi)$, where the coloring $\psi$ is defined by
% \begin{equation}
% \label{eq:psi}
% \psi(c) = \max_{a\in c} \phi(a),  
% \end{equation}
%  where $c\in C(\sigma \circ \pi)$.

The following result was conjectured by Stanley \cite{Stanley-preprint2006} and
proved by Féray \cite{F'eray-preprint2006}
and therefore we refer to it as Stanley-Féray character formula. 
For a more elementary proof we refer to \cite{F'eray'Sniady-preprint2007}.
\begin{theorem}
% [The original formulation of Stanley-Féray character formula]
\label{theo:stanley-feray-old}
The value of the normalized character on $\pi\in \Sym(n)$ for a
multirectangular Young diagram $\p\times\q$\/ for\/ $\p=(p_1,\dots,p_r)$,
$\q=(q_1,\dots,q_r)$ is given by 
\begin{equation}
\label{eq:stanley-feray-old}
%  \Sigma^{\p\times \q}_{\pi}=(-1)^n \sum_{(\sigma,\phi)\in
% \Sym(n)^{(r)}}
% \left[\prod_{b\in C(\sigma)} p_{\phi(b)} \prod_{c\in C(\sigma\circ \pi)}
% -q_{\psi(c)}\right],  
% 
% 
\Sigma^{\p\times \q}_{\pi}=
\sum_{\substack{\sigma_1,\sigma_2\in\Sym(n)\\ \sigma_1 \circ \sigma_2=\pi}}
\ \sum_{\phi_2:C(\sigma_2)\rightarrow\{1,\dots,r\}}
(-1)^{\sigma_1} \left[\prod_{b\in C(\sigma_1)} q_{\phi_1(b)} \prod_{c\in
C(\sigma_2)}
p_{\phi_2(c)}\right],
\end{equation}
where $\phi_1:C(\sigma_1)\rightarrow\{1,\dots,r\}$ is defined by 
$$ \phi_1(c) = \max_{\substack{b\in C(\sigma_2), \\ \text{$b$ and $c$
intersect}}} \phi_2(b).$$

\end{theorem}
% The left-hand side of the above equation makes sense if
% $p_1,p_2,\dots,q_1,q_2,\dots$ are natural numbers hence $\p\times\q$ is a
% Young diagram, but the right-hand side can be used to define the left-hand
% side for general multirectangular Young diagrams.

Interestingly, the above theorem shows that some partial information about the
family of graphs $(\V^{\sigma_1,\sigma_2})_{\sigma_1,\sigma_2}$ can be
extracted from the coefficients of Stanley polynomial $\Sigma^{\p\times
\q}_{\pi}$. This observation will be essential for the proof of the main result.

The following result is a simple corollary from Theorem
\ref{theo:stanley-feray-old} and it was proved by Féray
\cite{F'eray-preprint2008}.
\begin{theorem}
\label{theo:stanley-feray-for-cumulants}
For any integers $k_1,\dots,k_l\geq 1$ the value of the cumulant
$\kappa^{\id}(\Sigma_{k_1},\dots,\Sigma_{k_l})$ evaluated at the Young diagram
$\p\times \q$ is given by
\begin{multline*}
\kappa^{\id}{}^{\ \p\times \q}(\Sigma_{k_1},\dots,\Sigma_{k_l}) =\\
\sum_{\substack{\sigma_1,\sigma_2\in\Sym(n)\\ \sigma_1 \circ \sigma_2=\pi \\  
\langle \sigma,\pi\rangle \text{ transitive}}}
\ \sum_{\phi_2:C(\sigma_2)\rightarrow\{1,\dots,r\}}
(-1)^{\sigma_1}
\left[\prod_{b\in C(\sigma_1)} q_{\phi_1(b)} \prod_{c\in C(\sigma_2)}
p_{\phi_2(c)}\right],
\end{multline*}
where $n=k_1+\cdots+k_l$ and $\pi$ is a fixed permutation with the cycle
structure $k_1,\dots,k_l$, for example
$\pi=(1,2,\dots,k_1)(k_1+1,k_1+2,\dots,k_1+k_2)\cdots $, and where $\phi_1$ is
as in Theorem \ref{theo:main}.
% (k_1+\cdots+k_{l-1}+1,\cdots,k_1+\cdots+k_l)$.
\end{theorem}

\section{Toy example: Quadratic terms of Kerov polynomials}
\label{sec:toyexample}
We are on the way towards the proof of Theorem \ref{theo:main} which,
unfortunately, is a bit technically involved. Before dealing with the
complexity of the general case we shall present in this section a proof of
Theorem \ref{theo:quadratic} which concerns a simplified situation in which we
are interested in quadratic terms of Kerov polynomials. This case is
sufficiently complex to show the essential elements of the complete proof of
Theorem \ref{theo:main} but simple enough not to overwhelm the Reader with
unnecessary difficulties.

We shall prove Theorem \ref{theo:quadratic} in the following equivalent form:
\begin{theorem}
\label{theo:quadratic-reformulated}
For all integers $l_1,l_2 \geq 2$ and $k\geq 1$ the derivative 
$$  \left. \frac{\partial^2}{\partial R_{l_1} \partial R_{l_2} } K_{k}
\right|_{R_2=R_3=\cdots=0}$$ 
is equal to the number of triples $(\sigma_1,\sigma_2,q)$ with the following
properties:
\begin{enumerate}[label=(\alph*)]
 \item \label{enum:quadratic-a} 
$\sigma_1,\sigma_2$ is a factorization of the
cycle; in other words $\sigma_1,\sigma_2\in \Sym(k)$ are such that $\sigma_1
\circ \sigma_2=(1,2,\dots,k)$;
 \item $\sigma_2$ consists of two cycles and $\sigma_1$ consists of $l_1+l_2-2$
cycles;
 \item 
\label{enum:quadratic-c} 
$\ell:C(\sigma_2)\rightarrow \{1,2\}$ is a bijective labeling of the two cycles
of $\sigma_2$;
 \item for each cycle $c\in C(\sigma_2)$ there are at least $l_{\ell(c)}$ cycles
of $\sigma_1$ which intersect nontrivially $c$.
\end{enumerate}
\end{theorem}
% % % % % % % % % % % % % % % % % % % % % % % % % % % % 
% \begin{theorem}[The main result, reformulated]
% \label{theo:main-reformulated-junk}
% Let $k\geq 1$ and let $n_1,\dots,n_r\geq 2$ be a sequence of integers. The
% derivative of Kerov polynomial
% $$  \left. \frac{\partial}{\partial R_{n_1} } \cdots
% \frac{\partial}{\partial R_{n_r} } K_{k} \right|_{R_2=R_3=\cdots=0}$$ 
% is equal to the number of triples
% $(\sigma_1,\sigma_2,\ell)$ with the following properties:
% \begin{enumerate}[label=(\alph*)]
%  \item \label{enum:a}
% $\sigma_1,\sigma_2$ is a factorization of the
% cycle; in other words $\sigma_1,\sigma_2\in \Sym(k)$ are such that $\sigma_1
% \circ \sigma_2=(1,2,\dots,k)$;
%  \item $|C(\sigma_2)|=r$;
%  \item $|C(\sigma_1)|+|C(\sigma_2)|=n_1+\cdots+n_r$;
%  \item \label{enum:d} $\ell:C(\sigma_2)\rightarrow \{1,\dots,r\}$ is a
% bijection;
% % we require that for every
% % color $i\in\{2,3,\dots\}$ there are exactly $s_i$ cycles of $\sigma_2$
% % with color $i$;
%  \item for every set $A\subset C(\sigma_2)$ which is nontrivial (i.e.,
% $A\neq\emptyset$ and $A\neq C(\sigma_2)$) we require that there are more than 
% $\sum_{i\in A} \big(  n_{\ell(i)} -1 \big) $ cycles of $\sigma_1$ which
% intersect
% $\bigcup A$.
% \end{enumerate}
% \end{theorem}
% % % % % % % % % % % % % % % % % % % % % % % % % % % % % % % %
\begin{proof}
Equation \eqref{eq:s-r} shows that for any polynomial function $\F$ on the set
of generalized Young diagrams
$$ \frac{\partial^2}{\partial R_{l_1} \partial R_{l_2}} \F  = 
\frac{\partial^2}{\partial S_{l_1}\partial S_{l_2}} \F  + 
(l_1+l_2-1) \frac{\partial}{\partial S_{l_1+l_2}}
\F,
$$
where all derivatives are taken at $R_2=R_3=\cdots=S_2=S_3=\cdots=0$.
Theorem \ref{theo:stanley-and-S} shows that the right-hand side is equal to
% \begin{equation}
% \label{eq:jest-sierpien}
$$\left[p_1 p_2 q_1^{l_1-1} q_2^{l_2-1}\right] \F(\p\times\q) + 
(l_1+l_2-1) \left[p_1 q_1^{l_1+l_2-1}\right] \F(\p\times\q). $$
% \end{equation}
Lemma \ref{lem:tozsamosci-wielomianow-stanleya} applied to the second summand
shows therefore that
\begin{equation}
\label{eq:nabula}
\frac{\partial^2}{\partial R_{l_1} \partial R_{l_2}} \F =\left[p_1 p_2
q_1^{l_1-1} q_2^{l_2-1}\right] \F(\p\times\q) - 
\left[p_1 p_2 q_2^{l_1+l_2-2}\right] \F(\p\times\q). 
\end{equation}
In fact, the above equality is a direct application of Corollary
\ref{coro:jak-policzyc-pochodne-po-R}, nevertheless for pedagogical reasons we
decided to present the above expanded derivation. In the following we shall use
the above identity for $\F=\Sigma_k$.

On the other hand, let us compute the number of the triples
$(\sigma_1,\sigma_2,\ell)$ which contribute to the quantity presented in Theorem
\ref{theo:quadratic-reformulated}. By inclusion-exclusion principle it is equal
to 
\begin{multline}
\label{eq:paskuda}
\big(\text{number of triples which fulfill
conditions \ref{enum:quadratic-a}--\ref{enum:quadratic-c}}\big)+ \\
(-1) \big(\text{number of triples for which the cycle $\ell^{-1}(1)$} \\
\shoveright{\text{intersects at most $l_1-1$ cycles of $\sigma_1$}\big) +} \\
(-1) \big(\text{number of triples for which the cycle $\ell^{-1}(2)$} \\
\text{intersects at most $l_2-1$ cycles of $\sigma_1$}\big). 
\end{multline}
At first sight it might seem that the above formula is not complete since we
should also add the number of triples for which the cycle $\ell^{-1}(1)$
intersects at most $l_1-1$ cycles of $\sigma_1$ and the cycle $\ell^{-1}(2)$
intersects at most $l_2-1$ cycles of $\sigma_1$, however this situation is not
possible since $\sigma_1$ consists of $l_1+l_2-2$ cycles and $\langle
\sigma_1,\sigma_2\rangle$ acts transitively.

By Stanley-Féray character formula \eqref{eq:stanley-feray-old} the first
summand of \eqref{eq:paskuda} is equal to 
\begin{equation}
\label{eq:rown-a}
(-1) \sum_{\substack{i+j=l_1+l_2-2,\\ 1\leq j}} \left[p_1 p_2 q_1^{i}
q_2^{j} \right] \Sigma_k^{\p\times \q},
\end{equation}
the second summand of \eqref{eq:paskuda} is equal to 
\begin{equation}
\label{eq:rown-b}
  \sum_{\substack{i+j=l_1+l_2-2,\\ 1\leq i\leq l_1-1}} \left[p_1
p_2 q_1^{j} q_2^{i} \right] \Sigma_k^{\p\times \q}, 
\end{equation}
and the third summand of \eqref{eq:paskuda} is equal to 
\begin{equation}
\label{eq:rown-c}
  \sum_{\substack{i+j=l_1+l_2-2,\\ 1\leq j\leq l_2-1}}
\left[p_1 p_2 q_1^{i} q_2^{j} \right] \Sigma_k^{\p\times \q}.
\end{equation}

We can apply Corollary \ref{coro:order-does-not-matter} to the
summands of \eqref{eq:rown-b}; it follows that
\eqref{eq:rown-b} is equal to 
\begin{equation}
 \label{eq:rown-d}
  \sum_{\substack{i+j=l_1+l_2-2,\\ 1\leq i\leq l_1-1}}
\left[p_1 p_2 q_1^{i} q_2^{j} \right] \Sigma_k^{\p\times \q}.
\end{equation}

It remains now to count how many times a pair $(i,j)$ contributes to
the sum of \eqref{eq:rown-a}, \eqref{eq:rown-b}, \eqref{eq:rown-d}. It is not
difficult to see that the only pairs which contribute are $(0,l_1+l_2-2)$ and
$(l_1-1,l_2-1)$, therefore the number of triples described in the formulation
of the Theorem is equal to the right-hand of \eqref{eq:nabula} which finishes
the proof.
\end{proof}

\section{Combinatorial lemmas}
\label{sec:combinatorial-lemmas}

Our strategy of proving the main result of this paper will be to start with the
number of factorizations described in Theorem \ref{theo:main} and to interpret
it as certain linear combination of coefficients of Stanley polynomials for
$\Sigma_k$. The first step in this direction is promising: Stanley-Féray
character formula (Theorem \ref{theo:stanley-feray-old}) shows that indeed
Stanley polynomial for $\Sigma_k$ encodes certain information about the 
geometry of the bipartite graphs $\V^{\sigma_1,\sigma_2}$ for all
factorizations. Unfortunately, condition \ref{enum:marriage} is quite
complicated and at first sight it is not clear how to extract the information
about the factorizations fulfilling it from the coefficients of Stanley
polynomials.

In this section we will prove three combinatorial lemmas: Corollary
\ref{coro:good_factorizations}, Corollary \ref{coro:good_numbers} and
Corollary \ref{coro:good_numbers_and_partitions} which solve this difficulty.

\subsection{Euler characteristic}
Let $\I$ be a family of some subsets of a given finite set $\X$. 
We define
% \begin{equation}
$$ 
\chi(\I)=\sum_{l\geq 1}  
\sum_{\substack{
%  C\neq \emptyset, \\ 
C=(C_1\varsubsetneq 
\cdots
\varsubsetneq C_l), \\ C_1,\dots,C_l\in\I }}\!\!\!\!\! (-1)^{l-1},  
$$
% \end{equation}
where the sum runs over all non-empty chains $C=(C_1\varsubsetneq 
\cdots \varsubsetneq C_l)$ contained in $\I$. In such a
situation we will also say that $C$ is $l$-chain and $|C|=l$.
Notice that family $\I$ gives rise to a simplicial complex $\K$ with
$l-1$-simplices corresponding to $l$-chains 
% $(C_1\varsubsetneq  \cdots \varsubsetneq C_l)$ 
contained in $\I$ and the above quantity $\chi(\I)$ is just the Euler
characteristic of $\K$. 

The following lemma shows that under certain assumptions this Euler
characteristic is equal to $1$; we leave it as an exercise to adapt
the proof to show a stronger statement that under the same assumptions $\K$ is
in fact contractible (we will not use this stronger result in this article).

\begin{lemma}
\label{lem:zle_zbiory}
Let $\I$ be a non-empty family with a property that  
\begin{equation}  
\label{eq:bad-family}
A\cap B \in \I \quad \text{ or }\quad  A\cup B \in \I \qquad
\text{
holds for all $A,B\in \I$.}
\end{equation}
 
Then 
$$ \chi(\I)= 1. $$
\end{lemma}
\begin{proof}
% [Proof of Lemma \ref{lem:zle_zbiory}]
Let $\X=\{x_1,\dots,x_n\}$. We define 
$$\I_k= \big\{ A\cup \{x_1,\dots,x_k\}: A\in I \big\}.$$ 
Clearly $\I_0=\I$ and $\I_n=\{ \X \}$ therefore $\chi(\I_n)=1$. 
It remains to prove that
$\chi(\I_{k-1})=\chi(\I_{k})$ holds for all $1\leq k \leq n$ and we shall do it
in the following.

Let us fix $k$. For an $l$-chain $C=(C_1\varsubsetneq  \cdots \varsubsetneq
C_l)$ contained in $\I_{k-1}$ we define 
$$\iota_k(C) = \big(C_1\cup \{x_k\}
\subseteq \cdots \subseteq C_l\cup \{x_k\}\big)$$ 
which is a chain contained in
$\I_k$. Notice that $\iota_k(C)$ is either an $l-1$-chain (if $C_{i+1}=C_i\cup
\{x_k\}$ for some $i$) or $l$-chain (otherwise).
With these notations we have
\begin{multline*}
\chi(\I_{k-1})=\sum_{C: \text{non-empty chain in } \I_{k-1}} (-1)^{|C|-1}
= \\
\sum_{D: \text{non-empty chain in } \I_{k}} \
\sum_{\substack{C: \text{non-empty chain in } \I_{k-1},\\ \iota_k(C)=D}}
(-1)^{|C|-1}.
\end{multline*}
In order to prove $\chi(\I_{k-1})=\chi(\I_k)$ it is enough now to show that
for any non-empty chain $D=(D_1\varsubsetneq  \cdots \varsubsetneq D_l)$
contained in $\I_k$ 
\begin{equation} 
\label{eq:euler-proof}
(-1)^{|D|-1} =\sum_{\substack{C: \text{non-empty chain in }
\I_{k-1},\\ \iota_k(C)=D}}
(-1)^{|C|-1}. 
\end{equation}

Let $1\leq p\leq l$ be the maximal index with a property that $D_p\notin
\I_{k-1}$; if no such index exists we set $p=0$. In the remaining part of
this paragraph we will show that $D_i\setminus\{x_k\}\in \I_{k-1}$ holds for all
$1\leq i\leq p$. Clearly, in the cases when $p=0$ or $i=p$ there is nothing to
prove. Assume that $i<p$ and $D_i\setminus\{x_k\}\notin \I_{k-1}$. Since $D_i\in
\I_k$ it follows that $x_k\in D_i\in \I_{k-1}$.
It is easy to check that an analogue of
\eqref{eq:bad-family} holds true for the family $\I_{k-1}$. We apply this
property for $A= D_i$ and $B=D_p\setminus\{x_k\}$ which results in a
contradiction since $A\cap B=D_i\setminus\{x_k\}\notin \I_{k-1}$ and $A\cup
B=D_p\notin \I_{k-1}$.

Let $1\leq q\leq l$ be the minimal index
with a property that $x_k\in D_q$; if no such index exists we set $q=n+1$.
Similarly as above we show that $D_i\in \I_{k-1}$ holds for all
$q\leq i\leq l$.

The above analysis shows that a chain $C$ contained in $\I_{k-1}$ such that
$\iota_k(C)=D$ must have one of the following two forms:
\begin{enumerate}
 \item if $C=(C_1,\dots,C_l)$ is a $l$-chain then there exists a number $r$
($p\leq r<q$) such that
$$ C_i = \begin{cases} D_i\setminus \{x_k\} & \text{for } 1\leq i\leq r, \\
                       D_i                  & \text{for } r<i\leq l, 
         \end{cases}
 $$

 \item if $C=(C_1,\dots,C_{l+1})$ is a $l+1$-chain then there exists a number
$r$
($p<r<q$) such that
$$ C_i = \begin{cases} D_i\setminus \{x_k\} & \text{for } 1\leq i\leq r, \\
                       D_{i-1}                  & \text{for } r<i\leq l+1. 
         \end{cases}
 $$
\end{enumerate}
There are $q-p$ choices for the first case and there are $q-p-1$ choices for
the second case and \eqref{eq:euler-proof} follows.
\end{proof}

% Before the proof we will present an application of this result.

\subsection{Applications of Euler characteristic}

\begin{corollary}
\label{coro:good_factorizations}
Let $\sigma_1,\sigma_2\in \Sym(k)$ be permutations such that $\langle
\sigma_1,\sigma_2\rangle$ acts transitively and let
$q:C(\sigma_2)\rightarrow\{2,3,\dots\}$ be a coloring with a property that 
$$ \sum_{i\in C(\sigma_2)} q(i) = |C(\sigma_1)| + |C(\sigma_2)|.$$
We define $\I$ to be a family of the sets $A\subset C(\sigma_2)$ with the
following two properties:
\begin{itemize}
 \item $A\neq\emptyset$ and $A\neq C(\sigma_2)$,
 \item there are at most $\sum_{i\in A} \big( q(i)-1 \big) $ cycles of
$\sigma_1$ which intersect $\bigcup A$.
\end{itemize}

Then
\begin{equation}
\label{eq:detektor-jedynki}
\sum_{l\geq 0} \sum_{\substack{ C=(C_1\varsubsetneq 
\cdots
\varsubsetneq C_l), \\ C_1,\dots,C_l\in\I }}
% \!\!\!\!\! 
(-1)^{l} = 
\begin{cases} 1 & \text{if\/ } \I=\emptyset, \\ 0 & \text{otherwise.}
\end{cases} 
\end{equation}
\end{corollary}
\begin{proof}
It is enough to prove that the family $\I$ fulfills the assumption of Lemma
\ref{lem:zle_zbiory}; we shall do it in the following.
For $A\subseteq C(\sigma_2)$ we define
$$ f(A)= \left( \text{number of cycles of $\sigma_1$ which intersect 
$\bigcup A$} \right) - \sum_{i\in A} \big( q(i)-1 \big). $$
In this way $A\in\I$ iff $A\neq \emptyset$, $A\neq C(\sigma_2)$ and 
$f(A)\leq 0$.

It is easy to check that for any $A,B\subseteq C(\sigma_2)$
\begin{equation}
\label{eq:superadditivity}
 f(A)+f(B)\geq f(A\cup B)+f(A\cap B) 
\end{equation}
therefore if $A,B\in \I$ then $f(A\cup B)\leq 0$ or $f(A\cap B)\leq 0$.
If $A\cap B\neq \emptyset$ and $A\cup B\neq C(\sigma_2)$ this finishes the
proof. Since $f(\emptyset)=f\big( C(\sigma_2) \big) =0 $ also the case when
either $A\cap B=\emptyset$ or $A\cup B=C(\sigma_2)$ follows immediately.

It follows that if $A,B\in\I$ and $A\cap B,A\cup B\notin\I$ then $A,B\neq
\emptyset$, $A\cap B=\emptyset$, $A\cup B=C(\sigma_2)$, $f(A)=f(B)=0$. The
latter equality shows that
\begin{multline*} \left( \text{number of cycles of $\sigma_1$ which intersect
$\bigcup
A$} \right) + \\ \left( \text{number of cycles of $\sigma_1$ which intersect
$\bigcup B$} \right) = |C(\sigma_1)|\end{multline*}
therefore each cycle of $\sigma_1$ intersects either $\bigcup A$ or $\bigcup B$
which contradicts transitivity.
% 
% 
% For any $A,B\subset C(\sigma_2)$ we have
% $$\sum_{i\in A} \big( q(i)-1 \big) + \sum_{i\in B} \big( q(i)-1 \big) =
% \sum_{i\in A\cup B} \big( q(i)-1 \big) + \sum_{i\in A\cap B} \big( q(i)-1
% \big)$$
% and
% \begin{multline*} 
% \left( \text{number of cycles of $\sigma_1$ which intersect $\bigcup
% A$} \right) + \\
% \left( \text{number of cycles of $\sigma_1$ which intersect $\bigcup B$}
% \right)\geq  \\
% \left( \text{number of cycles of $\sigma_1$ which intersect $\bigcup (A\cup
% B)$} \right)+\\ 
%   \left( \text{number of cycles of $\sigma_1$ which intersect $\bigcup (A\cap
% B)$} \right).
% \end{multline*}
%  
% 
% 
\end{proof}

\begin{lemma}
For any $n\geq 1$
\begin{equation}
\label{eq:stirling}
 \sum_{k} (-1)^k\ \stirling{n}{k}\ k!  = (-1)^{n}, 
\end{equation}
where $\stirling{n}{k}$ denotes the Stirling symbol of the first kind, namely
the number of ways of partitioning $n$-element set into $k$ non-empty classes.
\end{lemma}
\begin{proof}
A simple inductive proof follows from the recurrence relation
$$\stirling{n}{k}= k\ \stirling{n-1}{k} + \stirling{n-1}{k-1}. $$
\end{proof}

\begin{corollary}
\label{coro:good_numbers}
Let $r\geq 1$ and let $k_1,\dots,k_r$ and $n_1,\dots,n_r$ be numbers such that
$k_1+\cdots+k_r=n_1+\cdots+n_r$.
We define $\I$ to be a family of the sets $A\subset \{1,\dots,r\}$ with the
following properties: $A\neq\emptyset$ and $A\neq\{1,\dots,r\}$ and
$$\sum_{i\in A} k_i \leq \sum_{i\in A} n_i.$$

Then
\begin{equation}
\label{eq:detektor-jedynki-2}
\sum_{l\geq 0} \sum_{\substack{ C=(C_1\varsubsetneq 
\cdots
\varsubsetneq C_l), \\ C_1,\dots,C_l\in\I }}
% \!\!\!\!\! 
(-1)^{l} = 
\begin{cases} (-1)^{r-1} & \text{if\/ } (k_1,\dots,k_r)=(n_1,\dots,n_r), \\ 0 &
\text{otherwise.}
\end{cases} 
\end{equation}
\end{corollary}
\begin{proof}
If $(k_1,\dots,k_r)=(n_1,\dots,n_r)$ then $\I$ consists of all subsets of
$\{1,\dots,r\}$ with the exception of $\emptyset$ and $\{1,\dots,r\}$.
Therefore there is a bijective correspondence between the chains
$C=(C_1\varsubsetneq \cdots \varsubsetneq C_l)$ which contribute to the
left-hand side of \eqref{eq:detektor-jedynki-2} and sequences
$(D_1,\dots,D_{l+1})$ of non-empty and disjoint sets such that
$D_1\cup \cdots \cup D_{l+1}=\{1,2,\dots,r\}$; this correspondence is defined
by requirement that
$$ C_i=D_1\cup\cdots\cup D_i.$$
It follows that the left-hand side of \eqref{eq:detektor-jedynki-2} is equal to 
$$ \sum_{l} (-1)^l\ \stirling{n}{l+1}\ (l+1)! $$
which can be evaluated thanks to \eqref{eq:stirling}.

We consider the case when $(k_1,\dots,k_r)\neq(n_1,\dots,n_r)$; for simplicity
we assume that $k_1\neq n_1$. 
We define 
$$f(A)=\sum_{i\in A} (k_i-n_i)$$
which fulfills \eqref{eq:superadditivity} and similarly as in the proof of
Corollary \ref{coro:good_factorizations} we conclude that condition
\eqref{eq:bad-family} is fulfilled under additional assumption that $A\cap B\neq
\emptyset$ or $A\cup B\neq\{1,2,\dots, r\}$; this means that Lemma
\ref{lem:zle_zbiory} cannot be applied directly and we must analyze the details
of its proof. We select the sequence $x_1,x_2,\dots$ used in the proof of Lemma
\ref{lem:zle_zbiory} in such a way that $x_1=1$. A careful inspection shows
that the proof of the equality $\chi(\I_0)=\chi(\I_1)$ presented above is still
valid. Since families $\I_1,\I_2,\dots$ fulfill condition \eqref{eq:bad-family}
therefore 
% $\chi(\I)=\chi(\I_0)=\
$\chi(\I_1)=\chi(\I_2)=\cdots=0 $.
\end{proof}

\begin{corollary}
\label{coro:good_numbers_and_partitions}
Let $r\geq 1$, let $\Pi\in P(1,2,\dots,r)$ be a partition, let
$n_1,\dots,n_r$ be numbers and let $\phi:\Pi\rightarrow\R$ be a
function on the set of blocks of the partition $\Pi$ with a property that
$$ \sum_{b\in\Pi} \phi(b) = n_1+\cdots+n_r$$
and $\phi(b)\geq |b|$ holds for each block $b\in\Pi$.
We define $\I$ to be a family of the sets $A\subset \{1,\dots,r\}$ with the
following properties: $A\neq\emptyset$ and $A\neq\{1,\dots,r\}$
and
$$\sum_{\substack{b\in\Pi,\\ b\cap A\neq \emptyset}} \left( \phi(b) -\big| b
\setminus A \big| \right)
 \leq \sum_{i\in A} n_i.$$

Then
\begin{equation}
\label{eq:detektor-jedynki-3}
\sum_{l\geq 0} \sum_{\substack{ C=(C_1\varsubsetneq 
\cdots
\varsubsetneq C_l), \\ C_1,\dots,C_l\in\I }}
% \!\!\!\!\! 
(-1)^{l} = 
\begin{cases} (-1)^{|\Pi|-1} & \text{if $\phi(b)=\sum_{i\in b} n_i$} \\
 &  \text{\ \ \ \ \ \ holds for each block $b\in\Pi$}, \\
0 & \text{otherwise.}
\end{cases} 
\end{equation}
\end{corollary}
\begin{proof}
We define
$$f(A)= |A|+\sum_{\substack{b\in\Pi,\\ b\cap A\neq\emptyset}} \left( \phi(b) -
|b| \right) $$ 
which fulfills \eqref{eq:superadditivity}. The remaining part of the proof is
analogous to Corollary \ref{coro:good_numbers}.
\end{proof}

\section{Proof of the main result}
\label{sec:proof}

We will prove Theorem \ref{theo:main} in the following equivalent form.
\begin{theorem}[The main result, reformulated]
\label{theo:main-reformulated}
Let $k\geq 1$ and let $n_1,\dots,n_r\geq 2$ be a sequence of integers. The
derivative of Kerov polynomial
$$  \left. \frac{\partial}{\partial R_{n_1} } \cdots
\frac{\partial}{\partial R_{n_r} } K_{k} \right|_{R_2=R_3=\cdots=0}$$ 
is equal to the number of triples
$(\sigma_1,\sigma_2,\ell)$ with the following properties:
\begin{enumerate}[label=(\alph*)]
 \item \label{enum:a}
$\sigma_1,\sigma_2$ is a factorization of the
cycle; in other words $\sigma_1,\sigma_2\in \Sym(k)$ are such that $\sigma_1
\circ \sigma_2=(1,2,\dots,k)$;
 \item $|C(\sigma_2)|=r$;
 \item $|C(\sigma_1)|+|C(\sigma_2)|=n_1+\cdots+n_r$;
 \item \label{enum:d} $\ell:C(\sigma_2)\rightarrow \{1,\dots,r\}$ is a
bijection;
% we require that for every
% color $i\in\{2,3,\dots\}$ there are exactly $s_i$ cycles of $\sigma_2$
% with color $i$;
 \item for every set $A\subset C(\sigma_2)$ which is nontrivial (i.e.,
$A\neq\emptyset$ and $A\neq C(\sigma_2)$) we require that there are more than 
$\sum_{i\in A} \big(  n_{\ell(i)} -1 \big) $ cycles of $\sigma_1$ which
intersect
$\bigcup A$.
\end{enumerate}
\end{theorem}
\begin{proof}
% From Corollary \ref{coro:good_factorizations} it follows that 
% 
% Let $$ X=(\underbrace{2,\dots,2}_{s_2\text{ times}},
% \underbrace{3,\dots,3}_{s_3\text{ times}},\dots ) $$
Let us sum both sides of \eqref{eq:detektor-jedynki} over all triples
$(\sigma_1,\sigma_2,\ell)$ for which conditions \ref{enum:a}--\ref{enum:d}
are fulfilled; for such triples we define the coloring
$q:C(\sigma_2)\rightarrow\{2,3,\dots\}$ by $q(i)=n_{\ell(i)}$.
It follows that the number of triples which fulfill all
conditions from the formulation of the theorem is equal to 
\begin{equation}
\label{eq:to-chcemy-policzyc}
\sum_{l\geq 0} \sum_{\substack{C=(C_1,\dots,C_l),\\
\emptyset\varsubsetneq C_1\varsubsetneq 
\cdots \varsubsetneq C_l\varsubsetneq \{1,2,\dots,r\} }}
% \!\!\!\!\! 
(-1)^{l}\ \Bad_C,   
\end{equation}
where $\Bad_C$ for $C=(C_1,\dots,C_l)$ denotes the number of triples
$(\sigma_1,\sigma_2,\ell)$ which fulfill \ref{enum:a}--\ref{enum:d} and such
that for each $1\leq j\leq l$ there are at most 
$\sum_{i\in C_j} \big(  n_{i} -1 \big) $ cycles of $\sigma_1$ which
intersect $\bigcup_{i\in C_j} \ell^{-1}(i)$.

Theorem \ref{theo:stanley-feray-old} shows that 
\begin{equation} 
\label{eq:bad}
\Bad_C = (-1)^{r-1} \sum_{k_1,\dots,k_r}  [p_1 \cdots p_r
q_1^{k_1-1} \cdots q_r^{k_r-1}] \Sigma^{\p\times\q}_k, 
\end{equation}
where the above sum is taken over
all integers $k_1,\dots,k_r\geq 1$ such that 
\begin{equation}
\label{eq:trudny-warunek-a}
% \left\{ \begin{split}
 k_1+\cdots+k_r=n_1+\cdots+n_r 
\end{equation}
% (MAM NIEJASNA NIEPEWNOSC, CZY OD KTOREJS STRON NIE NALEZALOBY KONSYSTENTNIE
% POODEJMOWAC JEDYNKI)
and
\begin{equation}
\label{eq:trudny-warunek-b}
\underbrace{k_{r+1-|C_j|}+\cdots+k_r
}_{|C_j| \text{ summands}} \leq
\sum_{i\in C_j} n_{i}  \qquad \text{holds for each $1\leq j\leq l$ }.  
%  \end{split}
% \right.
\end{equation}

We apply Lemma \ref{lem:tozsamosci-wielomianow-stanleya} to the right-hand side
of \eqref{eq:bad}. Therefore
\begin{equation}
\label{eq:BadC}
\Bad_C = (-1)^{r-1} \sum_{\Pi\in P(1,2,\dots,r)} \sum_{k_1,\dots,k_r} 
\NN^{\Sigma_k}_{\left( \left( \sum_{i\in b} k_i ,
|b|\right)  : b\in \Pi \right)}, 
\end{equation}
where the first sum runs over all partitions $\Pi$ of the set $\{1,2,\dots,r\}$
and the second sum runs over the tuples $k_1,\dots,k_r$ which fulfill 
conditions \eqref{eq:trudny-warunek-a}, \eqref{eq:trudny-warunek-b} and such
that the set of indices $i$ such that $k_i\geq 2$ coincides with the set of
rightmost legs of the blocks of $\Pi$. 
% 
% [WYSTARCZY OGRANICZYC SUMOWANIE DO BOOLOWSKICH PARTYCJI?]
% 
% SOME REALLY SMART ARGUMENTS SHOULD APPEAR HERE 
% 
% This last assumption implies that
% condition \eqref{eq:trudny-warunek-b} is equivalent to the following one:
% \begin{equation}
% \underbrace{k_{r+1-|C_j|}+\cdots+k_r
% }_{|C_j| \text{ summands}} \leq
% \sum_{i\in C_j} n_{i}  \qquad \text{holds for each $1\leq j\leq l$ such that
% }. 
% \end{equation}

For simplicity, before dealing with the general case, we shall analyze first the
contribution of the trivial partition  which consists only of singletons. We
define $\Bad^{\trivial}_C$ to be the expression \eqref{eq:BadC} with the
sum over partitions replaced by only one summand for $\Pi=\big\{
\{1\},\{2\},\dots,\{r\} \big\}$, i.e.
\begin{equation}
\label{eq:manipulacje-zlym}
\Bad_C^{\trivial} = (-1)^{r-1} \sum_{k_1,\dots,k_r} 
\NN^{\Sigma_k}_{(k_1,1),\dots,(k_r,1)}, 
\end{equation}
where the sum runs over the same set as in Equation \eqref{eq:bad} with an
additional restriction $k_1,\dots,k_r\geq 2$.
%  or, in other words,
% $$ \Bad_C^{\text{trivial}}= (-1)^{k+r} \sum_{k_1,\dots,k_r}
% N^{\Sigma_k}_{(k_1,1),\cdots,(k_r,1)} $$

Corollary \ref{coro:order-does-not-matter} shows that we may change the order
of the elements in the sequence $(k_1,\dots,k_r)$ hence
\eqref{eq:manipulacje-zlym} holds true also if the sum on the right-hand side
runs over all integers $k_1,\dots,k_r\geq 2$ such that
$k_1+\cdots+k_r=n_1+\cdots+n_r$
and such that for each $1\leq j\leq l$ 
$$ \sum_{i\in C_j} k_i \leq
\sum_{i\in C_j} n_{i}.   $$
Therefore for an analogue of the sum \eqref{eq:to-chcemy-policzyc} Corollary
\ref{coro:good_numbers} shows that
$$\sum_{l\geq 0} \sum_{\substack{C=(C_1,\dots,C_l),\\ \emptyset\varsubsetneq
C_1\varsubsetneq 
\cdots \varsubsetneq C_l\varsubsetneq \{1,2,\dots,r\} }}
% \!\!\!\!\! 
(-1)^{l}\ \Bad_C^{\text{trivial}}=(-1)^{r-1}
\NN^{\Sigma_k}_{(n_1,1),\cdots,(n_r,1)}. $$
Notice the the right-hand side is the summand appearing in Corollary
\ref{coro:jak-policzyc-pochodne-po-R} for the trivial partition $\Pi$ which is
quite encouraging.

Having in mind the above simplified case let us tackle the general partitions
$\Pi$. Corollary \ref{coro:order-does-not-matter} shows that we may shuffle the
blocks of partition $\Pi$ hence from \eqref{eq:BadC} it follows that
\begin{equation}
\label{eq:BadC-final-fantasy}
\Bad_C = (-1)^{r-1} \sum_{\Pi\in P(1,2,\dots,r)} \sum_{\phi} 
\NN^{\Sigma_k}_{\left( \big( \phi(b) ,
|b| \big)  : b\in \Pi \right)}, 
\end{equation}
where the second sum runs over all functions $\phi$ which assign integer
numbers to blocks of $\Pi$ and such that:
\begin{itemize}
 \item $\phi(b)\geq |b|+1$ holds for every block $b\in\Pi$; 
 \item $\sum_{b\in\Pi} \phi(b)=n_1+\cdots+n_r$;
%  \item $ \left[ \sum_{\substack{b\in\Pi,\\ b \cap C_j \neq
% \emptyset }}
% \begin{cases} 
% \phi(b)     & \text{if } b\subseteq C_j, \\
% |b\cap C_j| & \text{otherwise} 
% \end{cases} \right]
% \leq \sum_{i\in C_j} n_{i}
% $ holds for each $1\leq j\leq l$.
 \item $$  \sum_{\substack{b\in\Pi,\\ b \cap C_j \neq
\emptyset }} \left( \phi(b) - \big|b\setminus C_j\big| \right)
\leq \sum_{i\in C_j} n_{i}
$$ holds for each $1\leq j\leq l$.
\end{itemize}
Therefore \eqref{eq:to-chcemy-policzyc} is equal to 
% \begin{equation}
% \label{eq:to-chcemy-policzyc}
$$ \sum_{\Pi\in P(1,2,\dots,r)}  \sum_{\phi} 
\NN^{\Sigma_k}_{\left( \big( \phi(b) ,
|b| \big)  : b\in \Pi \right)} \left[ \sum_{l\geq 0}
\sum_{\substack{C=(C_1,\dots,C_l),\\
\emptyset\varsubsetneq C_1\varsubsetneq 
\cdots \varsubsetneq C_l\varsubsetneq \{1,2,\dots,r\} }}
\!\!\!\!\! 
 (-1)^{l+r-1}  \right]. $$
% \end{equation}
Corollary \ref{coro:good_numbers_and_partitions} can be used to calculate the
expression in the bracket hence the above sum is equal to 
$$ \sum_{\Pi\in P(1,2,\dots,r)}   
\NN^{\Sigma_k}_{\left( \big( \sum_{i\in b} n_i  ,
|b| \big)  : b\in \Pi \right)} 
 (-1)^{r-|\Pi|}.  $$
Corollary \ref{coro:jak-policzyc-pochodne-po-R} finishes the proof.
% \begin{multline*} \sum_{l\geq 0} \sum_{\substack{C=(C_1,\dots,C_l),\\
% \emptyset\varsubsetneq C_1\varsubsetneq 
% \cdots \varsubsetneq C_l\varsubsetneq \{1,2,\dots,r\} }}
% % \!\!\!\!\! 
% (-1)^{l}\ \Bad_C = \\ 
% \sum_{\Pi\in P(1,2,\dots,r)} 
% \end{multline*}
\end{proof}

Proof of Theorem \ref{theo:rattan-sniady-true} is analogous (the reference to
Theorem \ref{theo:stanley-feray-old} should be replaced by Theorem
\ref{theo:stanley-feray-for-cumulants}) and we skip it.

\section{Graph decomposition} 
\label{sect:compar}
In this section we will compare our main result with the previous complicated
combinatorial description of the coefficients of Kerov's polynomials proposed by F\'eray
in \cite{F'eray-preprint2008}. This will lead us to a new proof of the main
result of this paper, Theorem \ref{theo:main} and Theorem
\ref{theo:rattan-sniady-true}.

\subsection{Reformulation of the previous result}

Let us consider the formal sum of the collection of graphs
$(\V^{\sigma_1,\sigma_2})_{\sigma_1,\sigma_2}$ over all factorizations
$\sigma_1 \cdot \sigma_2 = (1,2,\ldots,k)$ (these graphs were defined in Section
\ref{subsec:defgraphs} but in order to be compatible with the notation of the 
paper \cite{F'eray-preprint2008} it might be more
convenient to allow multiple edges connecting two cycles with the multiplicity
equal to the number of the elements in the common support). Let us apply the
local transformations presented on
Figure \ref{fig:local_transformations} (the reader can easily imagine the
generalizations of the the drawn transformation to bigger loops: for a given
oriented loop of length $2k$ we remove in $2^k-1$ ways all non-empty subsets of
the set of edges oriented from a black vertex to a white vertex with the plus
or minus sign depending if the number of removed edges is odd or even) to each
of the summands and let us iterate this procedure until we obtain a formal
linear combination of forests. Of course, the final result $S$ may depend on
the choice of the loops used for the transformations, so in order to have a
uniquely determined result we have to choose the loops in some special way, for
example as described in paper \cite{F'eray-preprint2008}, the details of which
will not be important for this article.

\begin{figure}[t]
\includegraphics{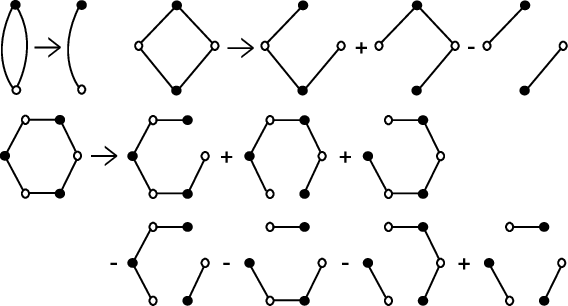} 
\caption{Local transformations on graphs.}
\label{fig:local_transformations}
\end{figure}

Then we have the following result:

\begin{theorem}[F\'eray \cite{F'eray-preprint2008}] \label{theo:feray_old}
The coefficient of $R_2^{s_2} R_3^{s_3} \ldots$ in $K_k$ is equal to
$(-1)^{1+s_2+s_3+\ldots}$ times 
the total sum of coefficients of all forests in
$S$ which consist of $s_i$ trees with one black and $i-1$ white vertices ($i$
runs over $\{2,3,\dots\}$).
%  and no other trees.
\end{theorem}

We will reformulate this result in a form closer to Theorem \ref{theo:main}. For
this purpose, if $(\sigma_1,\sigma_2,q)$ is a triple verifying conditions
\ref{enum:a}--\ref{enum:d} and $F$ is a subforest of
$\V^{\sigma_1,\sigma_2}$ with the same set of vertices, we will say that $F$ 
is a $q$-forest if the following two conditions are fulfilled:
\begin{itemize}
 \item all cycles of $\sigma_2$ (black vertices) are in different connected
components,
 \item each cycle $c$ of $\sigma_2$ is the neighbor of exactly $q(c)-1$ cycles
of $\sigma_1$ (white vertices).
\end{itemize}

\begin{theorem}
\label{theo:feray_reformule}
Let $k\geq 1$ and let $s_2,s_3,\dots$ be a sequence of non-negative integers
with only finitely many non-zero elements. The coefficient of the monomial
$R_2^{s_2} R_3^{s_3} \cdots $ in the Kerov polynomial $K_{k}$ is equal to
the number of triples $(\sigma_1,\sigma_2,q)$ which fulfill conditions
\ref{enum:a}--\ref{enum:d} of Theorem \ref{theo:main} and such that
\begin{enumerate}[label=(\eeeee)]
%  \item $\sigma_1,\sigma_2$ is a factorization of the
% cycle; in other words $\sigma_1,\sigma_2\in \Sym(k)$ are such that $\sigma_1
% \circ \sigma_2=(1,2,\dots,k)$;
%  \item the number of cycles of $\sigma_2$ is equal to the
% number of\/ factors in
% the product $R_2^{s_2} R_3^{s_3} \cdots $; in other words
% $|C(\sigma_2)|=s_2+s_3+\cdots$;
%  \item the total number of cycles of $\sigma_1$ and $\sigma_2$ is equal to the
% degree of the product
% $R_2^{s_2} R_3^{s_3} \cdots $; in other words
% $|C(\sigma_1)|+|C(\sigma_2)|=2 s_2+3 s_3+4 s_4+\cdots$;
%  \item $q:C(\sigma_2)\rightarrow \{2,3,\dots\}$ is a coloring of the cycles of
% $\sigma_2$ with a property that each color $i\in\{2,3,\dots\}$ is used exactly
% $s_i$ times;
% we require that for every
% color $i\in\{2,3,\dots\}$ there are exactly $s_i$ cycles of $\sigma_2$
% with color $i$;
 \item \label{enum:eprim} when we apply the transformations from Figure
\ref{fig:local_transformations} as prescribed in \cite[Section
3]{F'eray-preprint2008}, in the resulting linear combination of forests there is
(exactly one) $q$-forest.
\end{enumerate}
\end{theorem}

It is easy to see that this theorem is a reformulation of Theorem
\ref{theo:feray_old}.
A priori, it might seem that in the theorem above we should count each triplet
$(\sigma_1,\sigma_2,q)$ with multiplicity equal to the number of $q$-forests
appearing in the result, but we will prove in Corollary \ref{coro:combi} that it
is always equal to $0$ or $1$.

Comparing Theorem \ref{theo:feray_reformule} with Theorem \ref{theo:main} we
may wonder if conditions \ref{enum:marriage} and \ref{enum:eprim} are
equivalent. We will prove their equivalence in the following section.

%\begin{proposition}
%If\/ $\V^{\sigma_1,\sigma_2}$ is connected then
%$$\ref{enum:eprim} \implies \ref{enum:marriage}.$$
%\end{proposition}

%\begin{proof}
%Suppose that $F$ is a forest appearing when we iterate local
%transformations (in a certain way) on $\V^{\sigma_1,\sigma_2}$ with a
%property that all black vertices of $F$ are in different connected components.
%This means that there is a sequence of subgraphs of $\V^{\sigma_1,\sigma_2}$:
%$$\V^{\sigma_1,\sigma_2} = G_0 \supseteq G_1 \supseteq \ldots \supseteq G_{l-1}
%\supseteq G_l =F$$
%such that, for all $j$, $G_{j+1}$ appears in the result of a local
%transformation applied on $G_j$. Furthermore as $G_0$ (respectively, $G_l$)
%fulfills condition \eqref{enum:Gconnect} (respectively,
%\eqref{enum:G'nonconnect}) from Lemma \ref{lem:ineq_stricte}, there is an index
%$i$ such that $G:=G_i$ and $G':=G_{i+1}$ fulfill the conditions of Lemma
%\ref{lem:ineq_stricte}, therefore
%$$N_{G_{i+1}}(A) < N_{G_i}(A).$$
%But, as $\V^{\sigma_1,\sigma_2} = G_0 \supseteq G_i$ and $G_{i+1}  \supseteq G_l
%=F$, we also have
%$$\sum_{c \in A} \big[ q(c)-1\big] = N_{F}(A) \leq N_{G_{i+1}}(A) < N_{G_i}(A)
%\leq N_{V^{\sigma_1,\sigma_2}}(A)$$
%which finishes the proof.
%\end{proof}

\subsection{Equivalence of conditions \ref{enum:marriage} and \ref{enum:eprim}}

In Section \ref{subsec:condition-e-transportation} we introduced the
notion of $q$-admissibility of a graph. Recall that if a graph $G$ is connected
then it is $q$-admissible if and only if it satisfies condition
\ref{enum:marriage-graphs} which is a reformulation of \ref{enum:marriage}.
Notice also that if $G$ contains no loops then it is $q$-admissible if and only
if it is a $q$-forest.
% is a collection of trees of the following form: one black vertex $c\in
% V_\bullet$ connected to $q(c)-1$ leaves.

\begin{lemma}
\label{lem:nb_adm_inv}
The sum of coefficients of $q$-admissible graphs $G$ 
multiplied by $(-1)^{\text{(number of connected components of $G$)}}$ 
in a formal linear combination of
bipartite graphs with a given set of vertices and labeling
$q:V_\bullet\rightarrow\{2,3,\dots\}$ does not change after performing any
transformation of the form presented on Figure \ref{fig:local_transformations}.
\end{lemma}
\begin{proof}
Let us choose some oriented loop $L$ in graph $G$ and let us denote by  $E$ the
set of edges which can be erased in the corresponding local transformation from
Figure \ref{fig:local_transformations}; in other words $E$ consists of every
second edge in the loop $L$.

Consider the convex polyhedron $P$ (without boundary) which is the set of all
positive solutions $(x_e)_{\text{$e$ is an edge of $G$}}$ to the system of
equations from condition \ref{enum:transportation}. 
% From the following on we shall assume that $P$ is non-empvty;
% the reader may check that in the opposite case the following proof requires
% only minor modifications.

If $f$ is a real function on the set of edges of $G$ and $v$ is a vertex of $G$
we define $\big(\Phi(f) \big)(v)$ to be the sum of values of $f$ on edges
adjacent to $v$. If $P$ is non-empty then its dimension is equal to the
dimension of $\ker \Phi$. It is a simple exercise to show that $\Imag\Phi$
consists of all functions on vertices of $G$ with a property that for each
connected component of $G$ the sum of values on black vertices is equal to the
sum of values on white vertices hence 
$$\dimension \Imag \Phi= (\text{number of vertices of $G$})- (\text{number of
components of $G$}).$$
 It follows from rank-nullity theorem that
\begin{multline} 
\label{eq:dimension}
\dimension P = \dimension \ker \Phi =  \\
(\text{number of edges of $G$}) - (\text{number of vertices of
$G$}) + \\ (\text{number of connected components of
$G$}). 
\end{multline}

For a positive solution $(x_e)$ of our system of equations and a real number
$t$ we define 
$$y_{e}=\begin{cases}
         x_{e} & \text{if } e \notin L,\\
	 x_{e} + t & \text{if } e \in (L \setminus E),\\
	 x_{e} - t & \text{if } e \in E.
        \end{cases}$$
which is also a solution. Let $t$ be the minimal positive number for which
$(y_e)$ is not positive. In this way we define a map $\Pi:(x_e)\mapsto (y_e)$.

For any non-empty $A\subseteq E$ we define $P_A$ to be the set of positive
solutions with a property that 
$$\forall_{e\in E}\quad  e\in A  \iff  x_e=\min_{f\in E} x_{f}. $$
Since the defining condition for $P_A$ can be written in terms of some equations
and inequalities it follows that $P_A$ is a convex polyhedron.
It is easy to check that
\begin{multline*} \Pi(P_A) = \big\{ (x_e) : \text{non-negative solution such
that } \\  \forall_{e: \text{ edge of G}}\quad  (x_e=0) \iff (e\in A) \big\}.
\end{multline*}
The latter set can be identified with the set of positive solutions for our
system of equations corresponding to the graph $G'=G\setminus A$.
It follows that 
\begin{multline}
\label{eq:dimension2}
  \dimension P_A = 1+ \dimension \Pi(P_A)= \\ 1+
\left(\text{number of edges of $(G\setminus A)$}\right) - 
 (\text{number of vertices of $G$}) +\\ \left(\text{number
of connected components of $(G\setminus A)$}\right),
\end{multline}
where the last equality is just \eqref{eq:dimension} applied to $G'=G\setminus
A$.

It is easy to see that $P=\bigsqcup_{A\neq\emptyset} P_A$ is a disjoint union
therefore we have the equality between the Euler characteristics:
$$ \chi(P)= \sum_{A\neq \emptyset} \chi(P_A) $$
which thanks to  \eqref{eq:dimension} and \eqref{eq:dimension2} shows that
\begin{multline}
\label{eq:euler-charcteristic}
 (-1)^{(\text{number of connected components of $G$})} \ [\text{$P$ is
non-empty}] = \\ \sum_{A\neq \emptyset} (-1)^{|A|-1}\
 (-1)^{(\text{number of components of $G\setminus A$})}\ [\text{$P_A$ is
non-empty}] , 
\end{multline}
where we use the convention that 
$$[\text{(condition)}]=\begin{cases} 1 & \text{if
(condition) is true}, \\ 0 & \text{otherwise}. \end{cases} $$

% The right-hand side 
% is a convex polyhedron; it 
Proposition \ref{prop:existence-of-solution} shows that $P$ (respectively,
$P_A$) is non-empty if and only if $G$ (respectively, $G\setminus A$) is
$q$-admissible therefore \eqref{eq:euler-charcteristic} is equivalent to
\begin{multline*}
 (-1)^{(\text{number of components of $G$})} \ [\text{$G$ is $q$-admissible}]  =
\\ \sum_{A\neq \emptyset} (-1)^{|A|-1}\
 (-1)^{(\text{number of components of $G\setminus A$})}\ [\text{$(G\setminus A)$
is $q$-admissible}] , 
\end{multline*}
which is the desired equality.
\end{proof}

\begin{corollary} \label{coro:combi}
Suppose that $(\sigma_1,\sigma_2,q)$ is a triple verifying the conditions \ref{enum:a}--\ref{enum:d} of
Theorem \ref{theo:main}. If we iterate local transformations from Figure
\ref{fig:local_transformations} on $\V^{\sigma_1,\sigma_2}$ until we obtain a
formal linear combination of forests (not necessarily choosing the loops as
prescribed in \cite{F'eray-preprint2008}) then the sum of coefficients of
$q$-forests in the result is equal to 
$$\begin{cases} (-1)^{1+s_2+s_3+\ldots} & \text{if condition
           \ref{enum:marriage} is fulfilled;}\\
          		   0 & \text{otherwise.} 
 \end{cases}$$

In the case when we perform the transformations as prescribed in \cite[Section
3]{F'eray-preprint2008}, the sign property of this decomposition
(\cite[Proposition 3.3.1]{F'eray-preprint2008}) implies
that there is exactly one $q$-forest (with the appropriate sign) in the
resulting sum if condition \ref{enum:marriage} is fulfilled and there are no
$q$-forests otherwise; in other words condition \ref{enum:marriage} is
equivalent to \ref{enum:eprim}.
\end{corollary}

The above corollary together with Theorem \ref{theo:feray_reformule} give
another proof of the main result of the paper, Theorem \ref{theo:main}.
Analogous results can be stated for
the situation presented in Theorem \ref{theo:rattan-sniady-true}.

\appendix

\section{Results obtained after this paper has been submitted for publication}

In this appendix we present results which became available after this paper has been submitted for publication. Thus
they are not contained in the version published in \emph{Advances in Mathematics}.

\subsection{Closed walk interpretation of condition \ref{enum:marriage}}

\begin{proposition}
Condition \ref{enum:marriage} is equivalent
to the following one:
\begin{enumerate}[label=(e$^*$)]
\item \label{enum:closed-walk}
it is possible to chose orientations on the edges of the bipartite graph
$\V^{\sigma_1,\sigma_2}$ in such a way that: 
\begin{itemize}
\item every white vertex has exactly one
outgoing edge and every black vertex $j\in C(\sigma_2)$ has exactly $q(j)-1$
incoming edges, \item if we interpret orientations of edges as directions of
one-way streets, there exists a closed walk in the graph such that every black
vertex is visited at least once. 
\end{itemize}
\end{enumerate}
\end{proposition}
\begin{proof}
There is a bijective correspondence between the arrangements of marriages as 
in Section \ref{subsec:condition-e} and the arrangements of orientations of
edges given as follows: for any pair $i\in C(\sigma_1)$, $j\in C(\sigma_2)$ of
connected vertices, if a boy $i$ is married to a girl $j$, we draw an oriented
edge from vertex $i$ to vertex $j$; otherwise we draw an oriented edge in the
opposite direction.

Assume that condition \ref{enum:marriage} holds true. Condition
\ref{enum:marriage-prim} shows that it is possible to arrange marriages; we
fix the corresponding orientations of the edges.
Let $j\in C(\sigma_2)$ be a black vertex and let $A\subseteq C(\sigma_2)$
(respectively, $B\subseteq C(\sigma_1)$) be the set of black (respectively,
white) vertices $j'$ with a property that there exists a walk from $j'$ to $j$.
% For every white vertex $i\in B$ there is a walk $(i,j',\dots,j)$;
% furthermore, vertex $j'$ is uniquely determined as the wife of $i$. This shows
% that $j'\in A$. 
It is easy to see that the set of husbands of $A$ is equal to $B$. Furthermore,
every vertex in $B$ is connected only to vertices from $A$, therefore it is not
possible arrange marriages so that the set of wives of $B$ is different from
$A$; therefore it is not possible to arrange marriages so that the set of
husbands of $A$ is different from $B$. From condition \ref{enum:marriage-prim}
it follows that $A=C(\sigma_2)$. In this way we proved that any two black
vertices can be connected by a walk. By combining such walks we get the desired
closed walk which visits every black vertex at least once.

Conversely, assume that condition \ref{enum:closed-walk} holds true and let
$A\varsubsetneq C(\sigma_2)$, $A\neq\emptyset$ be a non-trivial subset. The
choice of orientations of the edges in the graph gives rise to some choice of
marriages. We denote by $B\subseteq C(\sigma_1)$ the set of husbands of $A$. In
the closed walk $(\dots,j,i,\dots)$ given by condition \ref{enum:closed-walk}
there must be a neighboring pair of vertices such that $j\in A\cup B$ and
$i\notin A\cup B$. It is easy to see that it is only possible if $j\in A$ and
$i\in C(\sigma_1)\setminus B$. This shows that the set of possible husbands for
$A$ contains  $B\sqcup\{i\}$ as a subset therefore condition \ref{enum:marriage}
is fulfilled. 
\end{proof}

\subsection{General formula for Kerov polynomials}

\begin{theorem}
\label{lem:extract-kerov}
Let $\mathcal{G}$ be a finite collection of connected bipartite graphs and let
$\mathcal{G}\ni G \mapsto m_G$ be a scalar-valued function on it. We assume that
$$ F(\lambda) = \sum_{G\in\mathcal{G}} m_G\ N_G(\lambda) $$
is a polynomial function on the set of Young diagrams; in other words $F$ can
be expressed as a polynomial in free cumulants.

Let $s_2,s_3,\dots$ be a sequence of non-negative
integers with only finitely many non-zero elements; then
$$ \left[ R_2^{s_2} R_3^{s_3} \cdots\right] F =
(-1)^{s_2+2s_3+3s_4+\cdots+1} \sum_{G\in\mathcal{G}} \sum_{q}  m_G,$$
where the sums runs over $G\in\mathcal{G}$ and $q$ such that:
\begin{enumerate}[label=(\alph*)]
 \addtocounter{enumi}{1}
 \item
\label{enum:ilosc2_Graphs}
the number of the black vertices of $G$ is
equal to $s_2+s_3+\cdots$;
 \item
\label{enum:boys-and-girls_Graphs}
the total number of vertices of $G$ is equal to $2 s_2+3 s_3+4 s_4+\cdots$;
 \item
\label{enum:kolorowanie_Graphs}
$q$ is a function from the set of
the black vertices to the set $\{2,3,\dots\}$; we require that each number
$i\in\{2,3,\dots\}$ is used exactly $s_i$ times;
\item
\label{enum:marriage_Graphs}
for every subset $A\subset V_\circ(G)$ of black vertices of $G$
which is nontrivial (i.e., $A\neq\emptyset$ and $A\neq V_\circ(G)$) there are
more than $\sum_{v\in A} \big( q(v)-1 \big)$ white vertices which are
connected to at least one vertex from $A$.
\end{enumerate}
\end{theorem}

In this paper we proved this result in the special case when
$F=\Sigma_n$ and $\mathcal{G}$ is the (signed) collection of bipartite maps
corresponding to all factorizations of a cycle, however it is not difficult to
verify that the proof presented in this article works without any modifications also in
this more general setup.

\section*{Acknowledgments}

Research of P{\'S} is supported by the MNiSW research grant P03A 013 30, by the
EU Research Training Network ``QP-Applications'', contract HPRN-CT-2002-00279 and
by the EC Marie Curie Host Fellowship for the Transfer of Knowledge ``Harmonic
Analysis, Nonlinear Analysis and Probability'', contract MTKD-CT-2004-013389.

PŚ thanks Marek Bożejko, Philippe Biane, Akihito Hora, Jonathan Novak,
Światosław Gal and Jan Dymara for several stimulating discussions during various
stages of this research project.

\bibliographystyle{alpha}

\bibliography{biblio2008}

\end{document}